\documentclass[12pt]{article}
\usepackage{amssymb}

\usepackage[utf8]{inputenc}
\usepackage{tikz}
\usetikzlibrary{matrix,arrows,decorations.pathmorphing}
\usetikzlibrary{decorations.pathreplacing} 
\usetikzlibrary{patterns}
\usepackage{hyperref}
\usepackage[utf8]{inputenc}
\usepackage[english]{babel}
\usepackage{booktabs}
\usepackage{amsthm,latexsym,amssymb,amsmath, amsfonts}
\usepackage{graphicx}
\usepackage{amsmath}
\usepackage{subcaption}
\usepackage{amsthm}
\usepackage{wrapfig}
\usepackage[margin=1in]{geometry}
\usepackage{bm}
\usepackage{mathtools}
\usepackage{caption}
\usepackage{lmodern,wrapfig,amsmath}
\usepackage{enumitem}
\usepackage{ stmaryrd }
\usetikzlibrary{shapes.misc}
\usetikzlibrary{cd}
\usepackage{adjustbox}
\usepackage{pgfplots}
\usepackage{tikz-3dplot}
\tdplotsetmaincoords{60}{115}
\pgfplotsset{compat=newest}
\usepackage{ amssymb }
\usepackage[backend=biber, style=ieee]{biblatex}
\bibliography{bibliography}
\usepackage{centernot}
\usepackage{faktor}
\usepackage{amsmath,amsfonts,amssymb, color, braket}

\newcommand{\adj}{\bigcup_{\mathcal{F}}}

\tikzcdset{scale cd/.style={every label/.append style={scale=#1},
    cells={nodes={scale=#1}}}}

\tikzset{
  curarrow/.style={
  rounded corners=8pt,
  execute at begin to={every node/.style={fill=red}},
    to path={-- ([xshift=50pt]\tikztostart.center)
    |- (#1) node[] {}
    -| ([xshift=-50pt]\tikztotarget.center)
    -- (\tikztotarget)}
    }
}

\newtheorem{theorem}{Theorem}[section]
\newtheorem{proposition}[theorem]{Proposition}

\newtheorem{definition}[theorem]{Definition}  
\newtheorem{lemma}[theorem]{Lemma}   

\newtheorem{remark}[theorem]{Remark}

\newtheorem{conjecture}[theorem]{Conjecture}

\title{A Non-Hausdorff de Rham Cohomology}
\author{David O'Connell}
\date{\small{\textit{Okinawa Institute of Science and Technology \\
1919-1 Tancha, Onna-son, Okinawa 904-0495, Japan}}}

\begin{document}

\maketitle

\begin{abstract}
In this paper we introduce and study the basic properties of de Rham cohomology for a certain class of non-Hausdorff manifolds. After a careful discussion of non-Hausdorff differential forms, we then provide a description of de Rham cohomology via Mayer-Vietoris sequences. We then use these sequences to prove both de Rham's Theorem and the Gauss-Bonnet theorem for our non-Hausdorff manifolds. Regarding the latter, we will see that the desired equality requires additional counterterms coming from the geodesic curvature of certain Hausdorff-violating submanifolds.
\end{abstract}

\section*{Introduction}
The Hausdorff property says that any pair of distinct points can be separated by disjoint open sets. In standard formulations of differential geometry, one assumes that all manifolds are Hausdorff spaces. This is normally justified on technical grounds: it can be shown that any locally-Euclidean, second-countable topological space admits partitions of unity subordinate to an arbitrary open cover, provided that the space is Hausdorff. These partitions of unity can then be used to build various global structures by patching together locally-defined ones. Conversely, any non-Hausdorff, locally-Euclidean second-countable space cannot admit a partition of unity subordinate to any cover by Hausdorff open sets, even when paracompactness is assumed. \\

Observations such as the above often render non-Hausdorff manifolds as too bothersome to include within standard treatments of differential geometry. Nonetheless, non-Hausdorff manifolds still occur within certain domains of mathematics, despite their technical inconvenience. Indeed, non-Hausdorff manifolds can be found within foliation theory and/or non-commutative geometry \cite{crainic1999remark,haefliger1957varietes, gauld2014non, deeley2022fell, kasparov1991groups, francis2023h, buss2012non}, and also within certain areas of mathematical physics \cite{luc2020interpreting, muller2013generalized, penrose1979singularities, earman2008pruning,hajicek1971causality, hajicek1970extensions, visser1995lorentzian, heller2011geometry, woodhouse1988geroch}  .  \\

In the seminal paper \cite{haefliger1957varietes}, Haefliger and Reeb observed that non-Hausdorff $1$-manifolds may naturally arise as the leaf spaces of certain foliations. Of special importance is their construction of non-Hausdorff manifolds via adjunction spaces. Loosely speaking, the authors showed that certain non-Hausdorff $1$-manifolds may be constructed by gluing together copies of the real line or along open subsets, whilst leaving the boundaries of these subsets unidentified, as pictured in Figure \ref{fig: line with two origins}. In a somewhat different context, Hajicek showed that any non-Hausdorff manifold is naturally covered by maximal Hausdorff open submanifolds \cite{hajicek1971causality}. These observations motivated the work of \cite{oconnell2023non}, which developed a general treatment of non-Hausdorff topological manifolds in terms of generalised adjunction spaces. In \cite{oconnell2023vector} this notion of adjunction space was extended to include smooth non-Hausdorff manifolds and the vector bundles fibred over them. \\

\begin{figure}
    \centering
    \begin{tikzpicture}
 
 \draw[] (-1,0)--(4,0);
 \draw[] (-1,-1)--(4,-1);
 \draw[dotted] (-0.5,0)--(-0.5,-1);
 \draw[dotted] (-0,0)--(-0,-1);
 \draw[dotted] (-0.5,0)--(-0.5,-1);
 \draw[dotted] (0.5,0)--(0.5,-1);
 \draw[dotted] (1,0)--(1,-1);
 \draw[dotted] (2,0)--(2,-1);
 \draw[dotted] (2.5,0)--(2.5,-1);
 \draw[dotted] (3,0)--(3,-1);
 \draw[dotted] (3.5,0)--(3.5,-1);

 \fill(1.5,0) circle[radius=0.06cm] {};
 \fill(1.5,-1) circle[radius=0.06cm] {};

 \draw[->] (4.5,-0.25) to[out=20, in=160] (6.5,-0.25);

 \draw[] (7,-0.5)--(9.4, -0.5);
 \draw[] (9.6, -0.5)--(12,-0.5);
 
 \fill(9.5,-0.25) circle[radius=0.06cm] {};
 \fill(9.5,-0.75) circle[radius=0.06cm] {};
 
\end{tikzpicture}
    \caption{\footnotesize{The line with two origins, the prototypical example of a non-Hausdorff manifold. Here it is constructed by gluing together two copies of $\mathbb{R}$ along the open subset $(\infty,0)\cup (0,\infty)$.}} 
    \label{fig: line with two origins}
\end{figure}
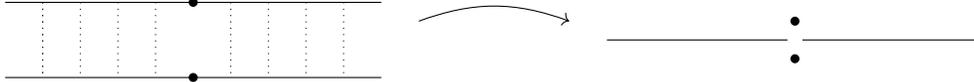

In this paper we will extend the formalism of \cite{oconnell2023non} and \cite{oconnell2023vector} once more by studying the differential forms that a non-Hausdorff manifold may admit. Motivated by Hajicek's original result, we will derive descriptions of non-Hausdorff de Rham cohomology via certain Mayer-Vietoris sequences. We will mostly follow the standard derivation of these sequences (found in \cite{bott1982differential}, say), with some key modifications ultimately due to certain technical difficulties surrounding the non-existence of arbitrary partitions of unity. After this, we use similar Mayer-Vietoris sequences for smooth singular cohomology in order to prove a non-Hausdorff version of de Rham's theorem. Finally, we will finish the paper with a discussion of the Gauss-Bonnet theorem for non-Hausdorff manifolds. \\ 

This paper is organised as follows. In Section 1 we will recall the basic formalism of non-Hausdorff manifolds, taken from \cite{oconnell2023non} and \cite{oconnell2023vector}. Included are the basic notions of colimits of smooth manifolds, descriptions of the algebras of smooth functions, and a discussion of non-Hausdorff vector bundles. Of central importance is the concept that sections of any bundle will satisfy a certain \textit{fibre product formula}. In this context, we mean that any section over a non-Hausdorff manifold can be uniquely described by a collection of sections defined over Hausdorff submanifolds, together with some consistency conditions on their mutual overlaps. We will spend due time on this observation, since it is a fundamental concept that underpins the remainder of the paper. \\

In Section 2 we will describe the differential forms on a non-Hausdorff manifold. Before doing so, we will first take some time to explain a description of vector fields in function-theoretic terms. This interpretation of vector fields needs particular care, since the type of bump functions used in locality arguments will generally not exist in our setting. Once this is done, we will describe the space of non-Hausdorff differential forms, as well as their exterior derivative. We then finish the section with a discussion of integration over non-Hausdorff manifolds. As a small aside, we will see that Stoke's theorem fails in a particularly controlled manner for our non-Hausdorff manifolds. \\

In Section 3 we will discuss the de Rham cohomology of non-Hausdorff manifolds. After making some important assumptions on the intersection properties of Hausdorff submanifolds, we will derive the aforementioned Mayer-Vietoris sequences for de Rham cohomology. In the case that a non-Hausdorff manifold is built from two Hausdorff spaces, this sequence is similar to the Hausdorff setting (cf. \cite{bott1982differential}). However, in the more-general case, we will derive a Čech-de Rham bicomplex that relates the de Rham cohomology of a non-Hausdorff manifold to a particular cover-dependent Čech cohomology (cf. \cite[§3]{oconnell2023vector}). \\

In Section 4 we will prove de Rham's theorem for non-Hausdorff manifolds. After some careful derivations of (smooth) singular homology, we may use our Mayer-Vietoris sequences together with and ``integration over chains" pairing to describe an isomorphism between de Rham cohomology and singular homology. Our approach will generally follow that of \cite{lee2013smooth}, with some important modifications due to the failure of Whitney's Embedding theorem in the non-Hausdorff case. \\

Finally, in Section 5 we will prove a non-Hausdorff version of the Gauss-Bonnet theorem for closed $2$-manifolds. As we will see, the integration formulas of Section 2, together with the Mayer-Vietoris sequences of Sections 3 and 4 will allow us to prove a Gauss-Bonnet theorem by decomposing the total scalar curvature into an integral over Hausdorff submanifolds. Applying the usual Gauss-Bonnet to each of these Hausdorff pieces will give the desired relationship. Although we can derive the result by reducing everything to the Hausdorff case, there is an important difference: for a non-Hausdorff $2$-manifold there will be additional contributions to curvature coming from the Hausdorff-violating submanifolds that sit inside it. \\

Throughout this paper we will assume that all manifolds, Hausdorff or otherwise, are locally-Euclidean second countable topological spaces. As a convention, we will use boldface to distinguish non-Hausdorff manifolds from their Hausdorff cousins. 

\section{Smooth Non-Hausdorff Manifolds}
We will now recall the basic topological properties of non-Hausdorff manifolds. The results here are taken from \cite{oconnell2023non} and \cite{oconnell2023vector}, so we will state them without proof.

\subsection{The Topology of non-Hausdorff Manifolds}
An adjunction of two topological spaces $X$ and $Y$ is formed from some subset $A$ of $X$ and a continuous map $f: A\rightarrow Y$. The adjunction space $X\cup_f Y$ is then defined to be the quotient of the disjoint union $X\sqcup Y$ by an equivalence relation that identifies every element in $A$ with its image under $f$. In categorical terms, the adjunction space $X\cup_f Y$ can be shown to be the pushout of the diagram \vspace{-3mm}
\begin{center}
    \begin{tikzcd}[row sep=3.5em]
A \arrow[r, "\iota_A"] \arrow[dr, phantom, "\lrcorner", very near start]\arrow[d, "f"'] & X \arrow[d, "\phi_X"] \\
Y \arrow[r, "\phi_Y"']                 & X\cup_f Y            
\end{tikzcd}
\end{center}
in the category $\textsf{TOP}$ \cite{brown2006topology, hatcher2002algebraic}, where here the maps $\phi_X$ and $\phi_Y$ send each point in $X$ or $Y$ into its corresponding equivalence class in the quotient space. \\

Motivated by this construction, we will now use this idea to form non-Hausdorff manifolds. To begin with, we would like to consider more general diagrams than the above. Following \cite{oconnell2023non}, we will consider diagrams consisting of manifolds, open submanifolds, and maps between them. Formally, we will consider a tuple $\mathcal{F}=(\textsf{M}, \textsf{A}, \textsf{f})$, where:
\begin{itemize}
    \item $\textsf{M}=\{ M_i \}_{i \in I}$ is a finite set of Hausdorff topological manifolds,
    \item $\textsf{A} = \{M_{ij}\}_{i,j \in I}$ is a set of bi-indexed open submanifolds, satisfying $M_{ij} \subseteq M_i$ for all $i,j$ in $I$, and
    \item $\textsf{f} = \{ f_{ij} \}_{i,j \in I}$ is a set of continuous maps $f_{ij}:M_{ij} \rightarrow M_j$. 
\end{itemize}
There is no guarantee that the data of $\mathcal{F}$ will yield a well-defined quotient space, so we will need to restrict our attention to the following triples.  
\begin{definition}\label{DEF: Adjunctive System}
The triple $\mathcal{F} = (\textsf{M}, \textsf{A}, \textsf{f})$ is called an adjunction system if it satisfies the following conditions for all $i,j \in I$.
\begin{enumerate}[itemsep=0.7mm]\setlength{\itemindent}{1em}
    \item[\textbf{A1)}] $M_{ii} = M_i$ and $f_{ii} = id_{M_i}$
    \item[\textbf{A2)}] $M_{ji} = f_{ij}(M_{ij})$, and $f_{ij}^{-1} = f_{ji}$
    \item[\textbf{A3)}]  $f_{ik}(x) = f_{jk} \circ f_{ij}(x)$ for each $x\in M_{ij}\cap M_{ik}$
\end{enumerate}
\end{definition}
\noindent 
The conditions listed above ensure that the resulting quotient space is well-defined. Given an adjunction space $\mathcal{F}$, we may define the adjunction space subordinate to it as:
$$ \textbf{M} := \adj M_i = \faktor{\left(\bigsqcup_{i\in I} M_i\right)}{\sim}$$
via the relation $\sim$ that identifies points $(x,i)$ and $(y,j)$ in the disjoint union whenever $f_{ij}(x) = y$. We will denote the points in $\textbf{M}$ by tuples $[x,i]$, where $x \in M_i$. Since $\textbf{M}$ is a quotient space, these points are actually equivalence classes of the form: 
$$  [x,i] := \Big\{ (y,j) \in \bigsqcup_{i\in I} M_i \ \big| \ y=f_{ij}(x)   \Big\}.  $$
We will denote by $\phi_i$ the continuous maps that send each $x$ in $M_i$ to its equivalence class $[x,i]$ in $\textbf{M}$. By construction, the open sets of $\textbf{M}$ can be usefully described with the following condition.

\begin{lemma}\label{LEM: open subsets in M}
    A subset $U$ of $\textbf{M}$ is open if and only if the set $\phi_i^{-1}(U)$ is open in $M_i$, for all $i$ in $I$.
\end{lemma}

Lemma 1.3 of \cite{oconnell2023non} ensures that the above adjunction space together with the maps $\phi_i$ satisfy a certain universal property. This allows us to view $\textbf{M}$ as the colimit of the diagram formed from the data in $\mathcal{F}$. 

In principle, quotienting some collection of manifolds may also spoil local charts. However, we may preserve the local nature of the manifolds $M_i$ by imposing some conditions on the gluing maps $f_{ij}$ and the submanifolds $M_{ij}$. The following result captures the finer details. 

\begin{lemma}\label{LEM: fij and Aij open implies phi maps open and M gen manifold}
Let $\mathcal{F}$ be an adjunction system in which the maps $f_{ij}$ are all open topological embeddings. Then
\begin{enumerate}
    \item the maps $\phi_i: M_i \rightarrow \textbf{M}$ are all open topological embeddings, and
    \item the adjunction space subordinate to $\mathcal{F}$ is locally-Euclidean and second-countable.
\end{enumerate}
\end{lemma}
The above result ensures that our adjunction spaces $\textbf{M}$ are locally-Euclidean, second-countable spaces, thereby justifying our notation. Since the canonical maps $\phi_i$ are now embeddings, we may also view the Hausdorff manifolds $M_i$ as genuine subspaces of $\textbf{M}$. To simplify notation, we will regularly identify $M_i$ with it's image $\phi_i(M_i)$, viewed as an open submanifold of $\textbf{M}$. \\

Lemma \ref{LEM: fij and Aij open implies phi maps open and M gen manifold} provides some conditions under which local charts of each $M_i$ may be preserved in the adjunction process. However, the Hausdorff property may or may not be preserved under these conditions.\footnote{Consider, for example, a binary adjunction in which $M_1 = M_2 = M_{12} = \mathbb{R}$, with $f_{12}=id$. The resulting adjunction space is $\mathbb{R}$. Alternatively, we may take $A = (-\infty, 0)$, and the adjunction space will be non-Hausdorff.} The key observation here revolves around the boundaries of the submanifolds $M_i$ -- if we identify $M_{ij}$ and $M_{ji}$ in the adjunction process, but leave their relative boundaries unidentified, then these boundaries (provided they exist) will violate the Hausdorff property in $\textbf{M}$. The following result summarises some useful facts surrounding this observation.

\begin{theorem}\label{THM: Summary of topological properties}
Let $\mathcal{F}$ be an adjunctive system that satisfies the criteria of Lemma \ref{LEM: fij and Aij open implies phi maps open and M gen manifold}, and let $\textbf{M}$ denote the adjunction space subordinate to $\mathcal{F}$. Suppose furthermore that each gluing map $f_{ij}: M_{ij} \rightarrow M_{ji}$ can be extended to a homeomorphism $\overline{f_{ij}}: Cl^{M_{i}}(M_{ij})\rightarrow Cl^{M_{j}}(M_{ji})$ such that the $\overline{f_{ij}}$ satisfy the conditions of Definition \ref{DEF: Adjunctive System}. Then:
\begin{enumerate}
    \item The Hausdorff-violating points in $\textbf{M}$ occur precisely at the $\textbf{M}$-relative boundaries of the subspaces $M_{ij}$.  
    \item $[x,i]$ and $[y,j]$ violate the Hausdorff property in $\textbf{M}$ if and only if $x \in \partial^{M_i} (M_{ij})$ and $y\in \partial^{M_j} (M_{ji})$ and $\overline{f_{ij}}(x) = y$.
    \item Each $M_i$ is a maximal Hausdorff open submanifold of $\textbf{M}$.
    \item $\textbf{M}$ is paracompact.
    \item If each $M_i$ is compact, then $\textbf{M}$ is compact.
\end{enumerate}
\end{theorem}

\noindent As mentioned in the introduction, non-Hausdorff manifolds will not admit partitions of unity subordinate to an arbitrary open cover. The following is taken from \cite[Thm 2.6]{oconnell2023non}. 

\begin{lemma}\label{LEM: no partitions of unity}
    Let $\mathcal{U}$ be an open set of a non-Hausdorff manifold $\textbf{M}$ such that each element of $\mathcal{U}$ is Hausdorff. Then there is not a partition of unity subordinate to $\mathcal{U}$. 
\end{lemma}

The key observation here is that any continuous function $\textbf{r}:\textbf{M}\rightarrow \mathbb{R}$ will necessarily map Hausdorff-violating points to the same value in $\mathbb{R}$. As such, any proposed partition of unity will either not sum to one everywhere, or will not be supported in the Hausdorff open sets of the proposed cover.\footnote{Note that there are at least \textit{some} open covers that admit partitions of unity -- the singleton $\{\textbf{M} \}$ being a somewhat trivial example.} 


\subsubsection*{Inductive Construction of $\textbf{M}$}
Generally speaking, many of the results that can be proved for binary adjunction spaces will also hold for a colimit of finitely-many manifolds. This is due to a certain inductive construction of colimits: if $\textbf{M}$ a non-Hausdorff manifold built from $(n+1)$-many manifolds $M_i$, then we can equivalently view $\textbf{M}$ as an adjunction of the two spaces $\textbf{N}$ and $M_{n+1}$, where: 
\begin{itemize}
    \item $\textbf{N}$ is the adjunction of the first $n$-many manifolds $M_i$, 
    \item the subset $A = \bigcup_{i\leq n} M_{i(n+1)}$, viewed as an open submanifold of $M_{n+1}$, and 
    \item the gluing map $F: A\rightarrow \textbf{N}$ sends each $a$ to the equivalence class in $\textbf{N}$ containing the point $f_{(n+1) i} (a)$, for all $i \leq n$.
\end{itemize}
The universal properties of colimits, applied to both $\textbf{M}$ and $\textbf{N}$, ensure the equivalence between these two representations of $\textbf{M}$. The details can be found in Section 1.1 of \cite{oconnell2023vector}.

\subsection{Smooth Manifolds}
We will now extend our discussion to the smooth setting. Suppose that we are given a non-Hausdorff manifold $\textbf{M}$ built according to Definition \ref{DEF: Adjunctive System} and Lemma \ref{LEM: fij and Aij open implies phi maps open and M gen manifold}.
Given atlases $\mathcal{A}_i$ of the Hausdorff manifolds $M_i$, we may define a topological atlas $\boldsymbol{\mathcal{A}}$ on $\textbf{M}$ as:  
$$ \mathcal{A} := \bigcup_{i \in I} \big\{ (\phi_i(U_\alpha), \varphi_\alpha \circ \phi_i^{-1}) \ | \ (U_\alpha, \varphi_\alpha) \in \mathcal{A}_i \big\}. $$
The antecedent conditions of Lemma \ref{LEM: fij and Aij open implies phi maps open and M gen manifold} ensure that the transition maps of $\mathcal{A}$ will be continuous for charts coming from different $M_i$. As one might expect, this observation similarly holds for smooth manifolds. The following result is taken from \cite[Lem 1.7]{oconnell2023vector}. 

\begin{lemma}\label{LEM: non-Haus smooth manifold}
Let $\textbf{M}$ be a non-Hausdorff manifold built according to Lemma \ref{LEM: fij and Aij open implies phi maps open and M gen manifold}. If, additionally, the $M_i$ are smooth manifolds and the $f_{ij}$ are all smooth maps, then $\textbf{M}$ admits a smooth atlas.
\end{lemma}

The atlas $\mathcal{A}$ described above will suffice as a smooth atlas of $\textbf{M}$. Under this reading, we may view the canonical maps $\phi_i : M_i \rightarrow \textbf{M}$ as smooth open embeddings. Moreover, arguments similar to that of \cite[Lem. 1.3]{oconnell2023non} ensure that $\textbf{M}$ satisfies a universal property, effectively turning it into the colimit of the diagram $\mathcal{F}$ in the category of smooth locally-Euclidean, second-countable spaces. \\

Smoothness of maps between non-Hausdorff manifolds can be defined as in the Hausdorff case, that is, by appealing to local coordinate representations about each point in the manifold. Since the submanifolds $M_i$ form an open cover of $\textbf{M}$, we see that any map $\boldsymbol{\psi}: \textbf{M}\rightarrow \textbf{N}$ is smooth if and only if all of the maps $\boldsymbol{\psi}\circ \phi_i: M_i \rightarrow \textbf{N}$ are.

\begin{remark}\label{REM: assumptions in this paper}
In addition to the conditions of Lemma \ref{LEM: fij and Aij open implies phi maps open and M gen manifold} and Theorem \ref{THM: Summary of topological properties}, hereafter we will also assume that our non-Hausdorff manifolds are smooth in the sense of Lemma \ref{LEM: non-Haus smooth manifold}.  Moreover, we will also need to assume that the gluing regions $M_{ij}$ have \textit{diffeomorphic boundaries}, which in this context means that the closures $Cl^{M_i}(M_{ij})$ are all smooth closed submanifolds with boundary, and the extended maps $\overline{f_{ij}}$ of Theorem \ref{THM: Summary of topological properties} are all diffeomorphisms from $Cl^{M_i}(M_{ij})$ to $Cl^{M_j}(M_{ji})$. In order to remove any potential ambiguity between the topological and manifold notion of boundary, we will also assume that the submanifolds $M_{ij}$ are all regular domains in the sense of \cite[Prop 5.46]{lee2013smooth}
\end{remark}

Smooth functions on a non-Hausdorff manifold $\textbf{M}$ can be defined in the same way as the Hausdorff case. In particular, the space $C^\infty(\textbf{M})$ of smooth functions is still a unital, associative algebra over $\mathbb{R}$. In our context, the canonical maps $\phi_i: M_i \rightarrow \textbf{M}$ are all smooth, so we might expect there to be some relationship between $C^\infty(\textbf{M})$ the algebras $C^\infty(M_i)$. As a matter of fact there is such a relationship, however in order to motivate the exact description we first need to make the following observation regarding the construction of functions on $\textbf{M}$, taken from \cite[§1.3]{oconnell2023vector}. 

\begin{lemma}\label{LEM: functions on M}
    Let $\textbf{M}$ be a smooth non-Hausdorff manifold satisfying the criteria of Remark \ref{REM: assumptions in this paper}. Then:
    \begin{enumerate}
        \item  If $r_i: M_i \rightarrow \mathbb{R}$ is a collection of smooth functions such that $r_i = r_j\circ f_{ij}$ for all $i,j$ in $I$, then the map $\textbf{r}: \textbf{M}\rightarrow \mathbb{R}$ defined by $\textbf{r}([x,i]) = r_i(x)$ is a smooth function on $\textbf{M}$.
        \item Any smooth function $r_i$ on $M_i$ can be extended to a smooth function $\textbf{r}$ on $\textbf{M}$.
    \end{enumerate}
\end{lemma}
Item (1) above is a direct reformulation of the so-called ``Gluing Lemma" for smooth maps, (cf \cite[Cor. 2.8]{lee2013smooth}). The second item above revolves around the so-called ``extension by zero" construction for Hausdorff manifolds (cf. \cite[Lem. 2.26]{lee2013smooth}). The idea behind the proof is to transfer the smooth function $r_i$ to each closed subset $\overline{M_{ji}}$ using the extended diffeomorphisms $\overline{f_{ij}}$, and then to inductively extend by zero into each $M_j$. This creates a collection of smooth functions defined on each $M_j$ that agree on their pairwise intersections, and the result then follows from Item (1) above. The results of Lemma \ref{LEM: functions on M} may be used to prove the following \cite[§1.3]{oconnell2023vector}. 

\begin{theorem}\label{THM: fibred product of smooth functions}
The algebra $C^{\infty}(\textbf{M})$ is isomorphic to the fibred product $$ \prod_{\mathcal{F}} C^\infty(M_i) :=\big\{ (r_1, ..., r_n) \in \bigoplus_{i\in I}  C^\infty(M_i) \big{|} \ r_i = r_j \circ f_{ij} \ on \ M_{ij} \ \text{for all } i,j\in I \big\}.      $$
\end{theorem}

This fibred-product structure has many useful consequences. Abstractly, it allows us to see $C^\infty(\textbf{M})$ as the limit of the diagram formed by applying the contravariant functor $C^\infty(\cdot)$ to the data of $\mathcal{F}$. Concretely, Theorem \ref{THM: fibred product of smooth functions} allows us to describe a smooth function $\textbf{r}$ on $\textbf{M}$ by any of the following equivalent descriptions: 
$$   \textbf{r} = (\phi_1^* \textbf{r}, \cdots , \phi_n^* \textbf{r}) = (\textbf{r}\circ \phi_1, \cdots , \textbf{r}\circ \phi_n) = (\textbf{r}|_{M_1}, \cdots , \textbf{r}|_{M_n}) = (r_1 ,\cdots , r_n).       $$
The above fibre product can be equivalently represented by a sequence. In the case of a binary adjunction space, we will have the following short sequence: 

$$  0 \rightarrow C^\infty(\textbf{M}) \xrightarrow{\ \ \Phi^* \ \ } C^\infty(M_1)\oplus C^\infty(M_2) \xrightarrow{ \ \iota_{12}^* - f_{12}^* \ } C^\infty(M_{12}) \rightarrow 0            $$
where here $\Phi^*$ is the map that sends any function $\textbf{r}$ to the pair $(\phi_1^*\textbf{r}, \phi_2^*\textbf{r})$, and $\iota_{12}:M_{12} \rightarrow M_1$ is the inclusion map. Theorem \ref{THM: fibred product of smooth functions} then states that the above sequence is exact at the first step, that is, $C^\infty(\textbf{M}) \cong \ker{(\iota_{12}^* - f_{12}^*)}$. However, due to Lemma \ref{LEM: no partitions of unity} there is no partition of unity subordinate to the open cover $\{M_1, M_2\}$, so in general we may not argue for the full exactness of this sequence by standard means, found in Prop. 2.3. of \cite{bott1982differential}, say. This is an important issue, the resolution of which will be the subject of Section 3.

\subsection{Vector Bundles}
Vector bundles fibred over non-Hausdorff manifolds $\textbf{M}$ can be constructed using colimits of Hausdorff bundles $E_i \xrightarrow{ \ \pi_i \ } M_i$. The construction is very similar to that of Section 1.1, except that we additionally require some bundle maps $F_{ij}$ covering the $f_{ij}$ that specify the gluing of each fibre. Schematically, the colimit $\textbf{E}$ of the bundles $E_i$ should make the diagram:

\begin{center}
    \begin{tikzcd}[row sep=1.9em, column sep=1.8em]
 & E_{ij} \arrow[ld, "F_{ij}"'] \arrow[dd, "\pi_{ij}"', near start] \arrow[rr, "I_{ij}"] &  & E_i \arrow[dd, "\pi_i"] \arrow[ld, "\chi_i"', dashed] \\
E_j \arrow[rr, "\chi_j", near end, crossing over, dashed] \arrow[dd, "\pi_j"'] &  & \textbf{E}  &  \\
 & M_{ij} \arrow[ld, "f_{ij}"', near start] \arrow[rr, "\iota_{ij}", near start] &  & M_i \arrow[ld, "\phi_i"] \\
M_j \arrow[rr, "\phi_j"'] &  & \textbf{M} \arrow[from=uu, "\boldsymbol{\pi}", dashed, near start, crossing over] & 
\end{tikzcd}
\end{center}

\noindent commute for all $i, j$ in $I$. This can all be confirmed formally, the details of which can be found in \cite[§2]{oconnell2023vector}. For the purposes of this paper, the relevant details are given below. 

\begin{theorem}\label{THM: colimit of bundles} 
Let $\mathcal{G}:= (\textsf{E}, \textsf{B}, \textsf{F})$ be a triple of sets in which: 
\begin{enumerate}
    \item $\textsf{E} = \{ (E_i, \pi_i, M_i) \}_{i \in I}$ is a collection of rank-$k$ vector bundles,
    \item $\textsf{B} = \{ E_{ij} \}_{i,j \in I}$ consists of the restrictions of the bundles $E_i$ to the intersections $M_{ij}$, and
    \item $\textsf{F} = \{F_{ij} \}_{i,j \in I}$ is a collection of bundle isomorphisms $F_{ij}: E_{ij} \rightarrow E_{ji}$ that cover the gluing maps $f_{ij}$ and satisfy the condition $F_{ik} = F_{jk}\circ F_{ij}$ on the intersections $M_{ijk}$, for all $i,j,k$ in $I$.
\end{enumerate}
Then the resulting adjunction space $\textbf{E}$ has the structure of a non-Hausdorff rank-$k$ vector bundle over $\textbf{M}$ in which the canonical inclusions $\chi_i: E_i \rightarrow \textbf{E}$ are bundle morphisms covering the canonical embeddings $\phi_i:M_i \rightarrow \textbf{M}$.
\end{theorem}

The bundle $\textbf{E}$ described above can be shown to satisfy a certain universal property, effectively turning it into a colimit of the bundles $E_i$. Moreover, it can be shown that every bundle $\textbf{E}\xrightarrow{\boldsymbol{\pi}} \textbf{M}$ is a colimit of Hausdorff bundles (cf \cite[§2.2]{oconnell2023vector}). As an application of this idea, we may readily obtain descriptions of the tensor bundles over the non-Hausdorff manifold $\textbf{M}$. The tangent bundle $T\textbf{M}$ can be constructed as a colimit of the bundles $TM_i$, where here the gluing maps for the fibres are given by the differentials of the gluing maps for $\textbf{M}$, i.e. $F_{ij}=df_{ij}$. According to Remark \ref{REM: assumptions in this paper} each $f_{ij}$ is a diffeomorphism, so their differentials will induce bundle isomorphisms on the $TM_{ij}$. The universal property of the colimit bundle can be applied (together with the maps $d\phi_i: TM_i \rightarrow T\textbf{M}$) to conclude that $T\textbf{M}$ is canonically isomorphic to a colimit of the bundles $TM_i$. According to Theorem \ref{THM: Summary of topological properties}, a pair of tangent vectors $[v,i]$ and $[w,j]$ in $T\textbf{M}$ will violate the Hausdorff property if and only if they satisfy $$ v \in T_p M_i \ \text{and} \ w\in T_{\overline{f_{ij}}(p)}M_j \ \text{and} \  d\overline{f_{ij}}(v) = w $$ for some point $p$ lying on the $M_i$-relative boundary of $M_{ij}$. Similarly, it can be shown that any higher tensor bundle $T^{(p,q)}\textbf{M}$ is canonically isomorphic to a colimit of the same bundles $T^{(p,q)}M_i$ defined on each $M_i$. 
\begin{theorem}\label{THM: tensor bundles canonical form}
    Let $\textbf{M}$ be a non-Hausdorff manifold satisfying the criteria of Remark \ref{REM: assumptions in this paper}. Then the rank $(p,q)$ tensor bundle $T^{(p,q)}\textbf{M}$ is canonically isomorphic to an adjunction of the bundles $T^{(p,q)}M_i$, glued along the maps $$F_{ij} := \underbrace{df_{ij} \otimes \cdots \otimes df_{ij}}_{p \ \text{times}}\otimes\underbrace{df_{ji} \otimes \cdots \otimes df_{ji}}_{q \ \text{times}} .  $$
\end{theorem}

We finish this section with a pair of very useful observations regarding the sections of non-Hausdorff bundles. To begin with, we observe that the sections of a non-Hausdorff colimit bundle $\textbf{E}$ can be described in a similar manner to that of smooth functions. 

\begin{theorem}\label{THM: sections are fibred product}
    Let $\textbf{E}$ be a vector bundle over $\textbf{M}$. Then the space of sections $\Gamma(\textbf{E})$ is isomorphic to the fibred product 
    $$ \prod_{\mathcal{F}} \Gamma(E_i) := \big\{ (s_1, ..., s_n) \in \bigoplus_{i\in I}  \Gamma(E_i) \big{|} \ s_i = s_j \circ f_{ij} \ on \ M_{ij} \ \text{for all } i,j\in I \big\}.  $$
\end{theorem}

\noindent In short, this means that any section on a bundle $\textbf{E}\xrightarrow{\pi} \textbf{M}$ canonically admits any of the following equivalent representations:
$$ \textbf{s} = (\phi_1^* \textbf{s}, \cdots , \phi_n^*\textbf{s} )= (\textbf{s}\circ \phi_1, \cdots , \textbf{s}\circ \phi_n ) = (\textbf{s}|_{M_1}, \cdots , \textbf{s}|_{M_n}) =  ( s_1, \cdots , s_n  ).   $$
Moreover, the isomorphism described in Theorem \ref{THM: sections are fibred product} is an isomorphism of $C^\infty(\textbf{M})$-modules, so the linear operations on the space $\Gamma(\textbf{E})$ translate into this representation as well.  \\

For our final observation, we see that sections can be used to describe the Hausdorff-violating points in a colimit bundle.

\begin{lemma}\label{LEM: Hausdorff-violating points for sections of bundles}
    Let $\textbf{E}$ be a vector bundle over $\textbf{M}$, with $\textbf{s}$ a section of $\textbf{E}$. If $[x,i]$ and $[y,j]$ are Hausdorff-inseparable in $\textbf{M}$, then $\textbf{s}([x,i])$ and $\textbf{s}([y,j])$ are Hausdorff-inseparable in $\textbf{E}$.
\end{lemma}
\begin{proof}
    Suppose first that $[x,i]$ and $[y,j]$ are Hausdorff-inseparable in $\textbf{M}$. Then there is some sequence $[a_\alpha, i]$ in $M_{ij}$ that converges to both points in $\textbf{M}$. Since any continuous map between second-countable spaces will preserve the convergence of sequences, the sequence $\textbf{s}([a_\alpha, i])$ will converge to both $\textbf{s}([x,i])$ and $\textbf{s}([y,j])$ in $\textbf{E}$. Since $\textbf{s}$ is injective, $\textbf{s}([x,i])$ and $\textbf{s}([y,j])$ are distinct points that are both limits of the same convergent sequence, so are Hausdorff-inseparable.
\end{proof}

\section{Non-Hausdorff Differential Forms}
In this section we will describe the differential forms on any non-Hausdorff manifold $\textbf{M}$ satisfying the criteria of Remark \ref{REM: assumptions in this paper}. As we will see, the fibre-product formula of Theorem \ref{THM: sections are fibred product} will provide a reasonable description of non-Hausdorff differential forms. However, special care is required to prove the locality of any derivative operators defined on $\textbf{M}$. To illustrate this idea, we will first provide a function-theoretic description of vector fields on $\textbf{M}$.

\subsection{Vector Fields}
Since non-Hausdorff manifolds are locally-diffeomorphic to Hausdorff ones, there are at least some definitions of tangent vectors which directly apply to our setting. Under a geometric interpretation, we can always define tangent spaces using velocity vectors of curves passing through a given point. However, a function-theoretic definition of non-Hausdorff tangent vectors is currently lacking in the literature. \\

Naively, one might try to define tangent vectors as point-derivations of the ring $C^\infty(\textbf{M})$. However, this will not make sense, since bump functions of the type used in \cite[Prop. 3.8]{lee2013smooth} to prove the locality of derivations will not exist in a non-Hausdorff manifold. Instead, we will consider derivations on the space $C^\infty_{[x,i]}(\textbf{M})$ of germs of functions at a point. Locality is naturally incorporated into this definition since $C^\infty_{[x,i]}(\textbf{M})$ will equal $C^\infty_{[x,i]}(U)$ for any open subset $U$ of $\textbf{M}$. We can therefore describe a basis of the tangent space using coordinate functions defined from open charts: $$ \sum_{a=1, \cdots, d} \lambda^a \frac{\partial}{\partial x_a}\Big{|}_{[x,i]} $$ as in the Hausdorff case (cf \cite[§14.1]{tu2011manifolds}). The geometric and function-theoretic definitions of tangent spaces can be shown to be equivalent, since they both have the same dimension as the tangent spaces of any $M_i$. An explicit isomorphism can be described using differentials $d\phi_i$, which either maps curves from $M_i$ into $\textbf{M}$, or maps the space $C^\infty_x(M_i)$ into $C^\infty_{[x,i]}(\textbf{M})$. \\

We may define the space $\mathfrak{X}(\textbf{M})$ of vector fields on $\textbf{M}$ as the global sections of the bundle $T\textbf{M}$. Locally, we see that a vector field $\textbf{X}$ in $\mathfrak{X}(\textbf{M})$ is smooth if and only if the coefficients $\lambda^a$ described above are elements of $C^\infty(U)$, where $U$ is the local chart used to induce the coordinate expression. By appealing to the standard notions of local charts (i.e. with the Euclidean Leibniz rule), we see that $\textbf{X}$ is indeed a derivation of the ring $C^\infty(\textbf{M})$. \\

In order to describe an explicit bijection between the space of vector fields and the space of derivations, we must first prove that every derivation is a local operator on $\textbf{M}$. We will approach this problem by first proving an auxilliary property, which we will call ``semi-locality". 
\begin{lemma}\label{LEM: derivations are semi-local}
       Let $\textbf{D}$ be some derivation of $C^\infty(\textbf{M})$. If $\textbf{r}$ is a smooth function on $\textbf{M}$ that vanishes on $M_i$, then $\textbf{D}\textbf{r}$ also vanishes on $M_i$.
    \end{lemma}
\begin{proof}
We first define a function $\boldsymbol{\psi}:\textbf{M} \rightarrow \mathbb{R}$ such that $\boldsymbol{\psi} \textbf{r} = \textbf{r}$ everywhere, with $\boldsymbol{\psi}=1$ for all Hausdorff-violating points in $M_i$. We construct this function in a few steps: first consider the set $B:= \bigcup_{j\neq i} \partial^i M_{ij}$, which consists of all Hausdorff-violating points in $\textbf{M}$ coming from $M_i$. Let $\boldsymbol{\psi}_i: M_i \rightarrow \mathbb{R}$ to be a smooth bump function for the set $B$, constructed as in \cite[§2]{lee2013smooth}. We now define the smooth map $\boldsymbol{\psi}:\textbf{M} \rightarrow \mathbb{R}$ as follows: 
$$   \boldsymbol{\psi}([x,j]) = \begin{cases} \psi_i \circ \phi_i^{-1}([x,j]) & \text{if} \ [x,j] \in M_i \\ 1 & \text{otherwise}   \end{cases}       $$
The map $\boldsymbol{\psi}$ will restrict to the bump function $\psi_i$ on $M_i$, and will act as a cutoff function that equals $1$ on $M_j$ everywhere outside the set $M_{ji}$. Figure \ref{FIG: bump function construction} depicts the function $\boldsymbol{\psi}$ for the branched real line. By construction, we have that $\textbf{r}\boldsymbol{\psi} = \textbf{r}$ everywhere on $\textbf{M}$. It follows that  $$ \textbf{D}\textbf{r} = \textbf{D}(\textbf{r}\boldsymbol{\psi}) = (\textbf{D}\textbf{r})\boldsymbol{\psi} + \textbf{r} (D\boldsymbol{\psi}),  $$
which reduces to the equality $\textbf{D}\textbf{r} = (\textbf{D}\textbf{r})\boldsymbol{\psi}$ on $M_i$. We will now show that $\textbf{D}\textbf{r}=0$ on $M_i$. There are two cases to consider. \\

Firstly, suppose that $[x,i]$ is some point in $M_i$ that does not violate the Hausdorff property in $\textbf{M}$. According to the construction of $\boldsymbol{\psi}$, this means that $\boldsymbol{\psi}([x,i])\neq 1$. Thus the Leibniz property tells us that $(\textbf{D}\textbf{r})([x,i]) = (\textbf{D}\textbf{r})([x,i])\cdot \boldsymbol{\psi}([x,i])$, hence $(\textbf{D}\textbf{r})([x,i]) = 0$. Suppose now that $[x,i]$ does violate the Hausdorff property in $\textbf{M}$. This means that $x$ lies in the set $B$, and thus $\boldsymbol{\psi}([x,i])=1$. Let $a_n$ be some sequence in $M_{ij}$ that converges to $x$ in $M_i$. We may always arrange for this sequence to avoid the other boundary sets in $B$, so without loss of generality we may conclude that $\boldsymbol{\psi}([a_n,i])\neq 1$ for all $n$. By assumption the function $\textbf{D}\textbf{r}$ is smooth, so it follows that $\textbf{D}\textbf{r}([a_n,i])$ converges to $\textbf{D}\textbf{r}([x,i])$ in $\mathbb{R}$. However, we just showed that $\textbf{D}\textbf{r}$ vanishes for every element of $M_i$ that does not violate the Hausdorff property. This means that $\textbf{D}\textbf{r}([a_n,i])$ is a constant sequence of zeroes in $\mathbb{R}$. The limit of this sequence is $0$, that is, $$\textbf{D}\textbf{r}([x,i]) = \textbf{D}\textbf{r}\left(\lim_{n\rightarrow \infty} [a_n,i]\right) =  \lim_{n\rightarrow \infty} \textbf{D}\textbf{r}([a_n,i]) = \lim_{n\rightarrow \infty} 0 = 0.$$   
We may thus conclude that $\textbf{D}\textbf{r}=0$ on $M_i$, as required.
    \end{proof}

\begin{figure}  
\centering 
  \begin{subfigure}[b]{0.3\linewidth}
    \centering
    \begin{tikzpicture}[scale=0.6]
 
 \draw[] (-1,0)--(4,0);
 \draw[] (-1,-1)--(4,-1);
 \draw[dotted] (-0.75,0)--(-0.75,-1);
 \draw[dotted] (-0.5,0)--(-0.5,-1);
  \draw[dotted] (-0.25,0)--(-0.25,-1);
 \draw[dotted] (-0,0)--(-0,-1);
  \draw[dotted] (0.25,0)--(0.25,-1);
 \draw[dotted] (0.5,0)--(0.5,-1);
  \draw[dotted] (0.75,0)--(0.75,-1);
 \draw[dotted] (1,0)--(1,-1);
 \draw[dotted] (1.25,0)--(1.25,-1);

 \fill(1.5,0) circle[radius=0.08cm] {};
 \fill(1.5,-1) circle[radius=0.08cm] {};

\draw[->] (1.5, 3)--(1.5,0.5);

 \draw[] (-1,3.5)--(1.45, 3.5);
 \draw[] (1.5, 3.4)--(4,2.75);
  \draw[] (1.5, 3.6)--(4, 4.25);
 
 \fill(1.5,3.4) circle[radius=0.08cm] {};
 \fill(1.5,3.6) circle[radius=0.08cm] {};
 
\end{tikzpicture}   
    \caption{\footnotesize{The $2$-branched real line}} \label{fig:M1}  
  \end{subfigure}
\begin{subfigure}[b]{0.3\linewidth}
   \centering
    \begin{tikzpicture}
        \draw[] (0.1,0)--(3.3,0);
        \draw[dotted] (1.71,0)--(1.71,3);   
        \fill(1.71,0) circle[radius=0.06cm] {};    
\begin{axis}[xmin=-3.2, xmax=3.2, ymin=-0.2, ymax=1.2, axis lines=none,
  scale=0.5]
\addplot[red, samples=100, smooth, domain=-0.95:-3, thick]
   plot (\x, { 0 });
\addplot[red, samples=100, smooth, domain=-1:1, thick, label={x}]
   plot (\x, {exp(1-1/(1-x^2)});
\addplot[red, thick, samples=100, smooth, domain=1:3]
   plot (\x, {0} );
\end{axis}
        
    \end{tikzpicture}
   
\caption{\footnotesize{The bump function on $M_1$}} \label{fig:M2}  
\end{subfigure}
\begin{subfigure}[b]{0.3\linewidth}
  \centering
   \begin{tikzpicture}
        \draw[] (0.1,0)--(3.3,0);
        \draw[dotted] (1.71,0)--(1.71,3);   
        \fill(1.71,0) circle[radius=0.06cm] {};    
\begin{axis}[xmin=-3.2, xmax=3.2, ymin=-0.2, ymax=1.2, axis lines=none,
  scale=0.5]
\addplot[red, samples=100, smooth, domain=-0.95:-3, thick]
   plot (\x, {0});
\addplot[red, samples=100, smooth, domain=-1:0, thick, label={x}]
   plot (\x, {exp(1-1/(1-x^2)});
\addplot[red, thick, samples=100, smooth, domain=0:3]
   plot (\x, {1} );
\end{axis}
        
    \end{tikzpicture}

\caption{\footnotesize{The cutoff function on $M_2$}} \label{fig:M2}  
\end{subfigure}
\caption{Schematics of the smooth map $\boldsymbol{\psi}$, defined for the $2$-branched real line.}
\label{FIG: bump function construction}
\end{figure}
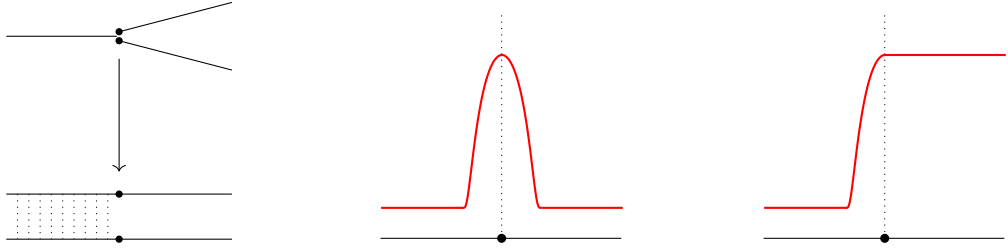

This notion of semi-locality means that we can always restrict a derivation $\textbf{D}$ of $C^{\infty}(\textbf{M})$ down to a derivation $D_i$ of $C^\infty(M_i)$ by defining $$ (D_i r)(x) = (\textbf{D}\textbf{r})([x,i]),$$ where $\textbf{r}$ is any extension of $r$ into $\textbf{M}$, à la Lemma \ref{LEM: functions on M}. Observe that $D_i$ inherits its linearity and Leibniz property from $\textbf{D}$. Moreover, the semi-locality property of Lemma \ref{LEM: derivations are semi-local} ensures that this map is well defined: any other extension $\tilde{\textbf{r}}$ of $r$ will cause the global function $\textbf{r} - \tilde{\textbf{r}}$ to vanish on $M_i$, which in turn means that $\textbf{D}\textbf{r} = \textbf{D}\tilde{\textbf{r}}$ on $M_i$. We will now argue that this notion of semi-locality can also be used to guarantee that an arbitrary derivation on $C^\infty(\textbf{M})$ is a local operator.
\begin{lemma}\label{LEM: derivations are local operators}
        Let $\textbf{D}: C^\infty(\textbf{M})\rightarrow C^\infty(\textbf{M})$ be a derivation. If $\textbf{r}$ is a smooth function that vanishes on some open set $U$, then $\textbf{D}\textbf{r}$ also vanishes on $U$.  
    \end{lemma}
\begin{proof}
        Let $p$ be some point in $U$. Since the $M_i$ cover $\textbf{M}$, $p$ is of the form $[x,i]$ for some $x$ in $M_i$. The restriction of the function $\textbf{r}$ to $M_i$ will vanish on the open set $U\cap M_i$. Since the restriction $D_i$ is a derivation on $C^\infty(M_i)$, and all Hausdorff derivations are local, it follows that $D_i f_i$ vanishes on the subset $U\cap M_i$. In the construction of the derivation $D_i$, we are free to choose the extension of $f_i$ to a global function in $\textbf{M}$. In particular, we may choose the somewhat-trivial extension $\textbf{r}$ to yield $ 0 = (D_i f_i)(x) = (\textbf{D}\textbf{r})([x,i]).$
    \end{proof}
With a confirmed notion of locality for derivations on $C^\infty(\textbf{M})$, an explicit isomorphism between the space of vector fields $\mathfrak{X}(\textbf{M})$ and the space of derivations $\text{Der}(C^\infty(\textbf{M}))$ can now be derived from the procedure outlined in \cite[Pg. 219]{tu2011manifolds}. Under this reading, we may meaningfully consider a vector field $\textbf{X}$ on $\textbf{M}$ as a derivation of $C^\infty(\textbf{M})$. Any vector fields on $M_i$ can be restricted to each $M_{ij}$ by an analogue to the above, so we can further restrict any derivation $D_i$ defined above to a derivation $D_{ij}$ on $C^\infty(M_{ij})$. 
In fact from Theorems \ref{THM: tensor bundles canonical form} and \ref{THM: sections are fibred product}, we may actually view any derivation $\textbf{D}$ of $C^\infty(\textbf{M})$ as an element of the fibred product $$\prod_{\mathcal{F}} \textrm{Der}(M_i) := \{ (D_1, \cdots, D_n) \ | \ D_i = D_j \ on \ M_{ij} \ \text{for all} \ i,j\in I \}.$$ In particular, this means that for any smooth function $\textbf{r}$ we have $\textbf{D}\textbf{r} = (D_1 r_1, \cdots , D_n r_n)$. 

\subsubsection*{The Lie Derivative}
In general, vector fields on $\textbf{M}$ will not admit global flows, ultimately due to the ``splitting" of vector fields at Hausdorff-violating points. However, any vector field on $\textbf{M}$ will be locally-equivalent to a Hausdorff vector field, so  there will still be unique local flows. As such, we may still meaningfully define the Lie derivative $\mathcal{L}_{\textbf{X}} \textbf{r}$ of a function $\textbf{r}$ with respect to a vector field $\textbf{X}$ pointwise as $ \mathcal{L}_{\textbf{X}} \textbf{r}([x,i]) = \mathcal{L}_{X_i} r_i(x). $ This is a manifesly local operator that will agree on the submanifolds $M_{ij}$, since both $\textbf{X}$ and $\textbf{r}$ do. As such, we may write Lie derivative as  
$$   \mathcal{L}_{\textbf{X}} \textbf{r} = (\mathcal{L}_{X_1} r_1, \cdots , \mathcal{L}_{X_n} r_n).   $$
As an immediate consequence, we observe that the equalities $\mathcal{L}_{\textbf{X}} \textbf{r} = \textbf{X}(\textbf{r})$ and $\mathcal{L}_{\textbf{X}} \textbf{Y} = [\textbf{X}, \textbf{Y}]$ both hold.   

\subsection{Differential Forms and the Exterior Derivative}
In similar spirit to the previous section, we may again apply Theorems \ref{THM: tensor bundles canonical form} and \ref{THM: sections are fibred product} to describe the algebra $\Omega^*(\textbf{M})$ of differential forms on $\textbf{M}$ via sections of the relevant exterior powers of $T^*\textbf{M}$. A fibred product formula will again hold for differential forms: given a vector field $\textbf{X}$ on $\textbf{M}$, a differential $1$-form $\boldsymbol{\omega}$ will act by $$ \boldsymbol{\omega}(\textbf{X}) = \left( \phi_1^*\omega(\phi_1^* X), \cdots, \phi_n^*\omega(\phi_n^* X)  \right) = (\omega_1(X_1) ,\cdots , \omega_n(X_n)) = \left( \boldsymbol{\omega}(\textbf{X})|_{M_1}, \cdots , \boldsymbol{\omega}(\textbf{X})|_{M_n} \right),      $$ 
\noindent with higher-degree forms acting analogously. Since the wedge product distributes over pullbacks, we may readily extend it to differential forms on $\textbf{M}$ using the equality: $$    \boldsymbol{\omega} \wedge \eta := (\omega_1 \wedge \eta_1, \cdots \omega_n \wedge \eta_n).$$
We may write an exterior derivative $d$ on $\Omega^*(\textbf{M})$ similarly: $$  d\boldsymbol{\omega} = (d\omega_1, \cdots, d\omega_n).  $$ Observe that this is well-defined, since $\iota_{ij}^*(d\omega_i) = d(\iota^*_{ij}\omega_i) = d(f_{ij}^*\omega_j) = f_{ij}^*(d\omega_j)$ for all $i,j$ in $I$. 
This definition of exterior derivative appears to be natural in that a semi-locality property similar to that of Lemma \ref{LEM: derivations are semi-local} is satisfied by construction. However, we ought to show that the operator $d$ uniquely satisfies the usual desiderata of the exterior derivative. 
\begin{theorem}\label{THM:exterior derivative properties}
The operator $\textbf{D}$ is the unique operator on $\Omega^*(\textbf{M})$ that satisfies the following properties.
\begin{enumerate}
            \item $d$ is local.
            \item $d\textbf{r}(\textbf{X}) = \textbf{X}(\textbf{r})$ for all $X\in \mathfrak{X}(\textbf{M})$ and $\textbf{r} \in C^\infty(\textbf{M})$.
            \item $d^2 = 0$
            \item $d(\boldsymbol{\omega} \wedge \boldsymbol{\eta}) = (d\boldsymbol{\omega}) \wedge \boldsymbol{\eta} + (-1)^q \boldsymbol{\omega} \wedge d\boldsymbol{\eta}$ for all $\boldsymbol{\omega} \in \Omega^q(\textbf{M})$ and $\boldsymbol{\eta} \in \Omega^*(\textbf{M})$.
        \end{enumerate}
    \end{theorem}
    \begin{proof}
        The locality of $d$ may be argued in a similar manner to that of Lemma \ref{LEM: derivations are local operators}. Indeed, suppose that we have some differential form $\boldsymbol{\omega}$ on $\textbf{M}$ that vanishes on an open subset $U$. For any $[x,i]$ in $U\cap M_i$ we have that $\boldsymbol{\omega}([x,i]) = \omega_i(x)$. In particular, $\omega_i$ will vanish on the open set $U\cap M_i$. Since the exterior derivative on $M_i$ is a local operator, we have that $d\boldsymbol{\omega} ([x,i]) = d\omega_i (x) = 0$. Suppose now that $\textbf{X}$ is some vector field on $\textbf{M}$. Then: $$(d\textbf{r})(\textbf{X}) = ((dr_1)(X_1), \cdots , (dr_n)(X_n)) = (X_1(r_1), \cdots X_n(r_n)) = \textbf{X}(\textbf{r}). $$
        Moreover, the exterior derivative on $\textbf{M}$ satisfies properties (iii) and (iv) above, since the exterior derivatives on $M_i$ do. To see that $d$ is unique with respect to these properties, suppose that there is some other operator $\tilde{d}$ on $\Omega^*(\textbf{M})$ satisfying items (i)--(iv) above. Since $\tilde{d}$ is local, we may restrict it to each $M_i$ to obtain an operator acting on $\Omega^*(M_i)$. The properties (ii)--(iv) are preserved under this restriction, and thus we may conclude that $d=\tilde{d}$ on each $M_i$. The global result then follows from the fact that $d\boldsymbol{\omega}([x,i]) = d\omega_i(x) = \tilde{d}\omega_i(x) = \tilde{d}\boldsymbol{\omega}([x,i])$ for all $i$ in $I$ and all $\boldsymbol{\omega}$ in $\Omega^*(\textbf{M})$.
     \end{proof}

\subsection{Integration on non-Hausdorff Manifolds}
Before getting to a definition of integration, we will first briefly comment on the existence of orientations on $\textbf{M}$. Globally speaking, an orientation can be seen as a non-vanishing top form, which is a section of the bundle $\bigwedge^d T^*(\textbf{M})$. According to Theorem \ref{THM: sections are fibred product}, any orientation on $\textbf{M}$ is induced from a series of compatible orientations on the submanifolds $M_i$. This idea can be restated as follows.
    \begin{lemma}\label{LEM: orientable manifold}
        Let $\textbf{M}$ be a colimit of Hausdorff manifolds, in accordance with \ref{REM: assumptions in this paper}. If each $M_i$ is oriented, and the gluing maps $f_{ij}$ are orientation-preserving, then $\textbf{M}$ admits an orientation under which the maps $\phi_i:M_i \rightarrow \textbf{M}$ become orientation-preserving, smooth open embeddings. 
    \end{lemma}
On a Hausdorff manifold, the integration of compactly-supported differential forms is typically defined in terms of local charts, and then the total result is patched together using a partition of unity. According to Lemma \ref{LEM: no partitions of unity} we do not have access to partitions of unity subordinate to arbitrary open covers, so we cannot use this definition in its desired generality. Instead, we will use the so-called ``integration over parameterisations". In this approach an overall integral is broken down into smaller pieces that intersect on measure zero sets. The following definition is an adaptation of that found in \cite[§16]{lee2013smooth}.
    \begin{definition}\label{DEF: Integration by parametrisation}
    Let $\textbf{M}$ be a $d$-dimensional non-Hausdorff manifold defined according to \ref{REM: assumptions in this paper} and oriented in the sense of \ref{LEM: orientable manifold}, and let $\boldsymbol{\omega}$ be a compactly-supported differential form on $\textbf{M}$. Suppose that we are given a finite collection $\{A_\alpha\}$ of open domains of integration in $\mathbb{R}^d$, together with a collection of maps $a_\alpha: \overline{A_\alpha}\rightarrow \textbf{M}$ such that: 
    \begin{enumerate}
        \item each $a_\alpha$ is an orientation-preserving diffeomorphism from $A_\alpha$ onto an open subset $U_{\alpha}$ in $\textbf{M}$, 
        \item the sets $U_\alpha$ are pairwise disjoint, and
        \item $\text{supp}(\boldsymbol{\omega}) := \overline{\{ [x,i] \in \textbf{M} \ | \ \boldsymbol{\omega}([x,i])=0 \} }\subseteq \bigcup_{\alpha} \overline{U_\alpha}$.
    \end{enumerate}
    We then define the integral of $\boldsymbol{\omega}$ in $\textbf{M}$ to be $$ \int_{\textbf{M}} \boldsymbol{\omega} = \sum_\alpha \int_{A_{\alpha}} a_i^*\boldsymbol{\omega}.   $$
    \end{definition}
Before deriving alternate descriptions of integration over $\textbf{M}$, we will first confirm that the above definition is well-defined. 
    \begin{proposition}\label{PROP: integration welldefined}
        The integral defined in Definition \ref{DEF: Integration by parametrisation} is independent of the particular choice of the parameterizing sets. 
    \end{proposition}
    \begin{proof}
        Suppose that $b_\beta:\overline{B_\beta}\rightarrow \textbf{M}$ is another collection of sets and maps satisfying the conditions of Definition \ref{DEF: Integration by parametrisation}, and denote by $V_\beta$ the image of $B_\beta$ under $b_\beta$. For each $U_\alpha$, we have that $$ \text{supp}(\boldsymbol{\omega}) \cap \overline{U_{\alpha}} \subseteq \overline{U_{\alpha}} \cap \bigcup_\beta \overline{V_{\beta}} = \bigcup_{\beta} \overline{U_{\alpha}} \cap\overline{V_{\beta}}.$$ As such, we may consider the sets of the form $A_{\alpha}\cap B_{\beta}$ as a parametrized set covering $A_\alpha$. This means that we can break down the integral of $\iota_\alpha^*\boldsymbol{\omega}$ over $A_\alpha$ into the sum: $$ \int_{A_{\alpha}} a_i^*\boldsymbol{\omega} = \sum_\beta \int_{A_{\alpha}\cap B_{\beta}} \iota_{\alpha\beta}^* a_i^*\boldsymbol{\omega}, $$ where the map $\iota_{\alpha\beta}$ is the inclusion of the set $A_{\alpha}\cap B_{\beta}$ into $A_\alpha$. We may repeat the same reasoning for fixed $V_\beta$ and break the integral of $b_\beta^*\boldsymbol{\omega}$ into a sum of integrals defined on the pairwise intersections $A_{\alpha}\cap B_{\beta}$. From this we may conclude that
        \begin{align*}
            \sum_\alpha \int_{A_{\alpha}} a_\alpha^*\boldsymbol{\omega} & =\sum_\alpha \sum_\beta \int_{A_{\alpha}\cap B_{\beta}} \iota_{\alpha\beta}^* a_\alpha^*\boldsymbol{\omega} \\
            & = \sum_\beta \sum_\alpha \int_{A_{\alpha}\cap B_{\beta}} \iota_{\beta\alpha}^* b_\beta^*\boldsymbol{\omega} \\
            & = \sum_\beta \int_{B_{\beta}} b_\beta^*\boldsymbol{\omega}
        \end{align*}
        as required.  
    \end{proof} 
        
Proposition 16.8 of \cite{lee2013smooth} ensures that Definition \ref{DEF: Integration by parametrisation} recovers the ordinary integral for Hausdorff manifolds. Moreover, by some arguments of point-set topology it can be shown that the pullback of a compactly-supported form by the maps $\phi_i$ will still be compactly-supported in each $M_i$ -- see Lemma \ref{LEM: restriction of compactly supported form is still compactly supported} in the Appendices for the details. We will now use this observation in conjunction with Definition \ref{DEF: Integration by parametrisation} to describe the total integral for $\textbf{M}$ by decomposing it into ordinary integrals defined over the submanifolds $M_i$.

   \begin{theorem}\label{THM: subadditivity for integration}
    Let $\textbf{M}$ be an oriented non-Hausdorff manifold, constructed as a binary adjunction of $M_1$ and $M_2$ according to \ref{REM: assumptions in this paper} and \ref{LEM: orientable manifold}. For all compactly-supported forms $\boldsymbol{\omega}$ on $\textbf{M}$, integration satisfies the sub-additive equality: $$ \int_{\textbf{M}} \boldsymbol{\omega} = \int_{M_1} \omega_1 + \int_{M_2} \omega_2 - \int_{\overline{M_{12}}} \omega_{12}$$ where here $\overline{M_{12}}$ denotes the closure of $M_{12}$ in $M_1$.
   \end{theorem}
    \begin{proof}
We would first like to describe a particular parametrisation of $\textbf{M}$ in terms of parametrisations of the $M_i$ and $M_{12}$. Suppose that:
    \begin{itemize}
        \item $\{a_i:A_i\rightarrow M_1\}$ is some parametrisation of the set $M_1\backslash M_{12}$, 
        \item $\{b_i:B_i\rightarrow M_2\}$ is some parametrisation of the set $M_2\backslash M_{12}$, and 
        \item $\{c_i:C_i\rightarrow \overline{M_{12}}\}$ is some parametrisation of the closed set $\overline{M_{12}}$, viewed as a subspace of $M_1$. 
    \end{itemize}  
   Observe that the collection $\{A_i\} \cup \{C_i\}$ forms a parametrisation of the entire manifold $M_1$, and similarly $\{B_i\} \cup \{\tilde{C}_i\}$ for $M_2$, where here $\tilde{C}_i$ denotes the transfer of the parametrization $C_i$  into $M_2$ using the diffeomorphism $\overline{f_{12}}$. In particular, this means that $$ \int_{M_1} \eta = \sum \left( \int_{A_i} a_i^* \eta \right) + \sum \left( \int_{C_i} ( \iota_{12}\circ c_i )^* \eta \right)= \sum \left( \int_{A_i} a_i^* \eta \right) + \sum \left( \int_{C_i}  c_i^* \circ\iota_{12}^* \eta \right)   $$
    for any differential form $\eta$ on $M_1$. Equivalently, this integral can be seen as a restriction of $\eta$ into the two sets $\overline{M_{12}}$ and $M_1 \backslash M_{12}$, followed by an integration over the respective parametrizations of these two spaces. We remark at this stage that the inclusion of the boundary of $M_{12}$ is necessary, as otherwise the restriction $\iota_{12}^*\eta$ need not be compactly-supported.\footnote{For example, suppose that $M_1 = \mathbb{R}$ and $M_{12} = (-\infty, 0)$, and consider a bump $1$-form in $M_1$ centered at the origin. The restriction of this form to $M_{12}$ will take support on some interval $[-a, 0)$, which is not compact in $M_{12}$. However, the interval $[-a,0]$ is compact in $\overline{M_{12}} = (\infty, 0]$.} Similarly, we have that: 
    $$ \int_{M_2} \mu = \sum \left( \int_{B_i} b_i^* \mu \right) + \sum \left( \int_{C_i} ( \overline{f_{ij}}\circ c_i)^* \mu \right) = \sum \left( \int_{B_i} b_i^* \mu \right) + \sum \left( \int_{C_i}  c_i^* \circ\overline{f_{ij}}^* \mu \right)    $$
    for any differential form $\mu$ on $M_2$. Since the maps $\phi_i$ are smooth, orientation-preserving open embeddings, the collection $$  \{\phi_1(A_i)\} \cup \{\phi_2(B_i)\} \cup \{\phi_1\circ \iota_{12}(C_i)\} $$ forms a parametrisation of $\textbf{M}$. According to Proposition \ref{PROP: integration welldefined}, the integral of any $\boldsymbol{\omega}$ can be computed according to this particular parametrization. This yields:
{\small
\begin{align*}
        \int_{\textbf{M}} \boldsymbol{\omega} & = \sum \int_{A_i} (\phi_1 \circ a_i)^*\boldsymbol{\omega} + \sum \int_{B_i} (\phi_2 \circ b_i)^*\boldsymbol{\omega} + \sum \int_{C_i} (  \phi_1  \circ \iota_{12}\circ c_i)^*\boldsymbol{\omega} \\
        &  \ \ \ \ \ \ \ \ + \sum \int_{C_i} ( \phi_1 \circ \iota_{12} \circ  c_i )^*\boldsymbol{\omega} -  \sum \int_{C_i} (  \phi_1 \circ \iota_{12}\circ c_i )^*\boldsymbol{\omega}  \\
        & = \sum \int_{A_i} (\phi_1 \circ a_i)^*\boldsymbol{\omega} + \sum \int_{B_i} (\phi_2 \circ b_i)^*\boldsymbol{\omega} + \sum \int_{C_i} (  \phi_1  \circ \iota_{12}\circ c_i)^*\boldsymbol{\omega} \\
        &  \ \ \ \ \ \ \ \ + \sum \int_{C_i} ( \phi_2 \circ \overline{f_{12}} \circ  c_i )^*\boldsymbol{\omega} -  \sum \int_{C_i} (  \phi_1 \circ \iota_{12}\circ c_i )^*\boldsymbol{\omega}  \\
        & = \left(\sum \int_{A_i} (\phi_1 \circ a_i)^*\boldsymbol{\omega} + \sum \int_{C_i} (  \phi_1  \circ \iota_{12}\circ c_i)^*\boldsymbol{\omega} \right) \\
        & + \left(\sum \int_{B_i} (\phi_2 \circ b_i)^*\boldsymbol{\omega} + \sum \int_{C_i} ( \phi_2 \circ \overline{f_{12}} \circ  c_i )^*\boldsymbol{\omega}   \right)  - \left( \sum \int_{C_i} (  \phi_1 \circ \iota_{12}\circ c_i )^*\boldsymbol{\omega}  \right) \\
        & = \int_{M_1} \phi_1^*\boldsymbol{\omega} + \int_{M_2} \phi_2^*\boldsymbol{\omega} - \int_{\overline{M_{12}}} \phi_{12}^*\boldsymbol{\omega}
    \end{align*}
}%
\noindent where here in the second equality we use the commutative property $\phi_1 \circ \iota_{12} = \phi_2 \circ \overline{f_{12}}$.
    \end{proof}

We can extend this representation to manifolds $\textbf{M}$ which are finite colimits. In this case, we will get a more-general alternating sum that keeps count of the number of sets being used in each possible intersection. As a shorthand notation, we will write $M_{i_1 \cdots i_p}$ for the set $M_1 \cap \cdots \cap M_p$. 
    \begin{theorem}\label{THM: Inductive form of integral}
        Let $\textbf{M}$ be an oriented non-Hausdorff manifold, built as the colimit of $n$-many manifolds $M_i$ in accordance with \ref{REM: assumptions in this paper}. The integral of a differential form $\boldsymbol{\omega}$ over $\textbf{M}$ satisfies: 
        $$ \int_{\textbf{M}} \boldsymbol{\omega} = \sum_{i=1}^n \left(\int_{M_i} \phi_i^*\boldsymbol{\omega}\right) -  \sum_{p=2}^n(-1)^{p}\left(\sum_{\substack{ i_1,\cdots, i_p \in I \\i_1 < \cdots < i_p}} \int_{\overline{M_{i_1 \cdots i_p}}} \phi_{{i_1 \cdots i_p}}^*\boldsymbol{\omega} \right)    $$
    \end{theorem}   
    \begin{proof}
        See Appendix A.1.
    \end{proof}

In the above arguments we had to include extra integrals over the boundaries of $M_{ij}$, since otherwise there is no guarantee that the restrictions $\boldsymbol{\omega}|_{M_{ij}}$ maintain a compact support. This may appear to be an innocuous detail, since the inclusion of a boundary component of $M_{12}$ will merely add an extra set of measure zero into the total integral. However, the inclusion of this boundary will have some-far reaching consequences regarding certain types of integral. To illustrate this, we present the following counterexample to Stoke's theorem. 
\begin{lemma}\label{LEM: Stoke theorem counterexample}
        Let $\textbf{M}$ be a binary adjunction such that $\textbf{M}$, and the $M_i$ are compact, without boundary. Then $$\int_{\textbf{M}} d\boldsymbol{\omega} = -  \int_{\partial M_{12}} \phi_{12}^*\boldsymbol{\omega}. $$
\end{lemma}   
    \begin{proof}
        Using Theorem \ref{THM: subadditivity for integration} together with Stoke's Theorem for Hausdorff manifolds, we have: 
        \begin{align*} 
        \int_{\textbf{M}} d\boldsymbol{\omega} & = \int_{M_1} \phi_1^*d\boldsymbol{\omega} + \int_{M_2} \phi_2^*d\boldsymbol{\omega} - \int_{\overline{M_{12}}} \phi_{12}^*d\boldsymbol{\omega} \\
        & = \int_{M_1} d\phi_1^*\boldsymbol{\omega} + \int_{M_2} d\phi_2^*\boldsymbol{\omega} - \int_{\overline{M_{12}}} d\phi_{12}^*\boldsymbol{\omega} \\ 
        & = \int_{\partial M_1} \phi_1^*\boldsymbol{\omega} + \int_{\partial M_2} \phi_2^*\boldsymbol{\omega} - \int_{\partial{M_{12}}} \phi_{12}^*\boldsymbol{\omega} \\
        & = - \int_{\partial{M_{12}}} \phi_{12}^*\boldsymbol{\omega}
        \end{align*}
        as required.
    \end{proof}
In the Hausdorff setting, this particular integral will vanish. However, we see here that in the non-Hausdorff integral the internal Hausdorff-violating submanifold still has an influence. As a consequence of the above, we see that exact differential forms may not integrate to zero on a closed non-Hausdorff manifold. 

\section{de Rham Cohomology via Mayer-Vietoris Sequences}
In this section we will describe the de Rham cohomology of a non-Hausdorff manifold $\textbf{M}$ satisfying the conditions of Remark \ref{REM: assumptions in this paper}. According to Theorem \ref{THM: Summary of topological properties}, the (images of) the Hausdorff manifolds $M_i$ sit inside $\textbf{M}$ as open sets. Moreover, we saw in Section 2 that the space of differential forms $\Omega^*(\textbf{M})$ and its exterior derivative $d$ can be described in terms of the same data defined on each $M_i$. With this in mind, we will now set about deriving Mayer-Vietoris sequences for the de Rham cohomology of $\textbf{M}$, computed according to the open cover $\{M_i \}$.

\subsection{A Binary Mayer-Vietoris Sequence}
We will derive our Mayer-Vietoris sequences inductively. So, throughout this section we fix a non-Hausdorff manifold $\textbf{M}$ which is constructed from a pair of Hausdorff spaces $M_1$ and $M_2$ according to Remark \ref{REM: assumptions in this paper}. According to our discussion in the previous section, the differential forms on $\textbf{M}$ display a fibred product structure. In a similar manner to our description of the smooth functions $C^\infty(\textbf{M})$ in Section 1.2, this can be arranged into the following sequence.
\begin{center}
\begin{tikzcd}
0 \arrow[r] & \Omega^q(\textbf{M}) \arrow[r, "\Phi^*"] & \Omega^q(M_1)\oplus \Omega^q(M_2) \arrow[r, "\iota_{12}^* - f_{12}^*"] & \Omega^q(M_{12}) \arrow[r] & 0
\end{tikzcd}
\end{center}
Ideally we would like these sequences to be exact, since then we could apply the Snake Lemma to form a long exact sequence on cohomology. Although the fibre product structure of each $\Omega^q(\textbf{M})$ guarantees that the above sequences are exact in the first term, without the existence of a partition of unity subordinate to $\{M_1, M_2\}$, there is no guarantee that the difference map $\iota_{12}^* - f_{12}^*$ is surjective. As such, the above sequences will not be short exact in general. To resolve this issue, we will instead exchange $\Omega^q(M_{12})$ for a subalgebra. Before doing so, we first make the following observation.
\begin{lemma}\label{LEM: fibred product for differential forms with boundary}
Let $\overline{M_{12}}$ be the closure of $M_{12}$ in $M_1$. Then there is an equality $$ \Omega^k(M_1)\times_{\Omega^{k}(M_{12})}\Omega^k(M_2) = \Omega^k(M_1)\times_{\Omega^{k}(\overline{M_{12}})}\Omega^k(M_2).   $$
\end{lemma}     
\begin{proof}
The inclusion from right-to-left is clear -- any pair of forms that agree on $\overline{M_{12}}$ will also agree on $M_{12}$. For the converse, suppose that $\omega_i$ are a pair of $q$-forms on $M_{i}$ such that $\iota_{12}^*\omega_1 = f_{12}^*\omega_2$ on $M_{12}$. By our assumption, it suffices to show that $\omega_1(x) = \omega_2(\overline{f_{12}}(x))$ for all elements $x$ on the boundary of $M_{12}$. So, consider a point $x$ in $\partial{M_{12}}$. By Theorem \ref{THM: Summary of topological properties}, the points $[x,1]$ and $[f_{ij}(x), 2]$ violate the Hausdorff property in $\textbf{M}$.  According to Lemma \ref{LEM: Hausdorff-violating points for sections of bundles}, the global form $\boldsymbol{\omega}$ constructed from $\omega_1$ and $\omega_2$ can be used to describe the Hausdorff-violating points of the $q$-form bundle over $\textbf{M}$. In our situation, this means that $\boldsymbol{\omega}([x,1])$ and $\boldsymbol{\omega}([f_{12}(x), 2])$ violate the Hausdorff property in the $q$-form bundle over $\textbf{M}$. Again by Theorem \ref{THM: Summary of topological properties}, this means that $\omega_1 = \overline{f_{12}}^* \omega_2$, that is, $$ \iota_{12}^*\omega_1(x) = \omega_1 \circ \iota_{12}(x) = \omega_1(x) = \omega_2(\overline{f_{12}}(x)) = \overline{f_{12}}^* \omega_2 (x), $$
as required.
\end{proof}

By restricting our attention to the differential forms on $M_{12}$ that are finite on its boundary, we allow ourselves to extend any differential form $\omega$ in $\Omega^q(\overline{M_{12}})$ by zero into either $M_1$ or $M_2$. This ensures that both the pullback maps $\iota^*$ and $f^*$ are surjective, and consequently so is the difference of the two. By Lemma \ref{LEM: fibred product for differential forms with boundary} we obtain our desired short exact sequence,
\begin{center}
\begin{tikzcd}
0 \arrow[r] & \Omega^k(\textbf{M}) \arrow[r, "\Phi^*"] & \Omega^k(M_1)\oplus \Omega^k(M_2) \arrow[r, "\iota_{12}^* - (\overline{f_{12}})^*"] & \Omega^k(\overline{M_{12}}) \arrow[r] & 0
\end{tikzcd}
\end{center}
\noindent and thus we may form the following long exact sequence for de Rham cohomology. 
\begin{center}
\adjustbox{scale=0.9}{\begin{tikzcd}[arrow style=math font,cells={nodes={text height=2ex,text depth=0.75ex}}]
\tikzset{
  curarrow/.style={
  rounded corners=8pt,
  execute at begin to={every node/.style={fill=red}},
    to path={-- ([xshift=50pt]\tikztostart.center)
    |- (#1) node[fill=white] {$\scriptstyle d^*$}
    -| ([xshift=-50pt]\tikztotarget.center)
    -- (\tikztotarget)}
    }
}      
       
       \cdots \arrow[r] & H_{dR}^q (\textbf{M}) \arrow[r, "\Phi^*"] & H_{dR}^q(M_{1}) \oplus  H_{dR}^q(M_2) \arrow[r,"\iota^*_{12} - (\overline{f_{12}})^*"] \arrow[draw=none]{d}[name=Y, shape=coordinate]{} & H_{dR}^q(\overline{M_{12}}) \arrow[curarrow=Y]{dll}{} & \\
       & H_{dR}^{q+1} (\textbf{M}) \arrow[r, "\Phi^*"] & H_{dR}^{q+1}(M_{1}) \oplus  H_{dR}^{q+1}(M_2) \arrow[r, "\iota^*_{12} - (\overline{f_{12}})^*"]  & H_{dR}^{q+1}(\overline{M_{12}}) \arrow[r] & \cdots 
\end{tikzcd}}
\end{center}
In general we may not proceed any further than this. However, we observe that if $\textbf{M}$ is built in such a way that the gluing region $M_{12}$ is regular-open, then we may derive a Mayer-Vietoris sequence that is perhaps more familiar. Recall that a (sub)manifold with boundary is homotopically equivalent to its interior (cf. \cite[Thm 9.26]{lee2013smooth}). Moreover, for a regular open set, it is always the case that the interior of its closure equals itself. Combining these two facts, we see that requiring $M_{12}$ to be regular-open amounts to requiring that $M_{12}$ is homotopically equivalent to its closure $\overline{M_{12}}$. Since de Rham is homotopy invariant for Hausdorff manifolds, our extra requirement ensures that $H^q_{dR}(M_{12})$ and $H^q(\overline{M_{12}})$ are isomorphic. This means that we can replace terms in the above sequence to obtain the following. 
\begin{center}
\adjustbox{scale=0.9}{\begin{tikzcd}[arrow style=math font,cells={nodes={text height=2ex,text depth=0.75ex}}]
\tikzset{
  curarrow/.style={
  rounded corners=8pt,
  execute at begin to={every node/.style={fill=red}},
    to path={-- ([xshift=50pt]\tikztostart.center)
    |- (#1) node[fill=white] {$\scriptstyle d^*$}
    -| ([xshift=-50pt]\tikztotarget.center)
    -- (\tikztotarget)}
    }
}      
       
       \cdots \arrow[r] & H_{dR}^q (\textbf{M}) \arrow[r, "\Phi^*"] & H_{dR}^q(M_{1}) \oplus  H_{dR}^q(M_2) \arrow[r,"\iota^*_{12} - f_{12}^*"] \arrow[draw=none]{d}[name=Y, shape=coordinate]{} & H_{dR}^q(M_{12}) \arrow[curarrow=Y]{dll}{} & \\
       & H_{dR}^{q+1} (\textbf{M}) \arrow[r, "\Phi^*"] & H_{dR}^{q+1}(M_{1}) \oplus  H_{dR}^{q+1}(M_2) \arrow[r, "\iota^*_{12} - f_{12}^*"]  & H_{dR}^{q+1}(M_{12}) \arrow[r] & \cdots 
\end{tikzcd}}
\end{center}

\subsection{The Čech-de Rham Bicomplex}
Suppose now that $\textbf{M}$ is covered by finitely-many $M_i$. In this case, we will proceed with an argument similar to that of \cite{bott1982differential, oconnell2023vector, crainic1999remark} to construct a bicomplex which will describe the de Rham cohomology of $\textbf{M}$. \\

Before getting to the desired result, we first need to comment on some additional assumptions required in the proceeding arguments. In what follows, we will supplement each higher intersection $M_{i_1 \cdots i_p}$ with its Hausdorff boundary, and compute the Mayer-Vietoris long exact sequence according to these sets. However, our derivation will not hold in the case that the closed submanifolds $\overline{M_{i_1 \cdots i_p}}$ carry additional cohomological information that is not already present in their interiors. In order to avoid such situations, we will assume that the higher intersections satisfy the closure property: $$   \overline{M_{i_1 \cdots i_m}} =  \bigcap_{p=1}^m \overline{M_{i_p}} $$ 
for all $m$ in $I$. Of course, this property does not hold for arbitrary closed sets. However, it removes the possibility of sets $M_{ij}$ intersecting only on the codimension-1 boundary components. \\ 

With these extra assumptions in mind, we will now derive a Mayer-Vietoris long exact sequence for $\textbf{M}$. In doing so, we will need to use the Čech differential $\delta$, the details of which can be found in \cite[§8]{bott1982differential}. In our context, the map $\delta$ generalises the difference maps $\iota_{ij}^* - f_{ij}^*$ as follows: 
$$  \delta \omega_{i_1 \cdots i_{p+1}} = \sum_{\alpha=1}^p (-1)^\alpha \iota^*_{i_1 \cdots i_{\hat{\alpha}} \cdots i_p} \omega_{i_1 \cdots i_{\hat{\alpha}} \cdots i_p},      $$
where here the map $\iota_{i_1 \cdots i_{\hat{\alpha}} \cdots i_p}: M_{i_1 \cdots i_p}\rightarrow  M_{i_1 \cdots i_{\hat{\alpha}} \cdots i_p}$ is the inclusion map, and the caret denotes the omission of that particular index. Note that by virtue of its combinatorics, the Čech differential always squares to zero (cf \cite[Prop. 8.3]{bott1982differential}). 

\begin{theorem}\label{THM: MV long exact sequence for de rham}Let $\textbf{M}$ be a colimit of $n$-many manifolds $M_i$ in accordance with \ref{REM: assumptions in this paper}. In addition, suppose that each $M_{ij}$ is regular-open in $M_i$ and that the closure-intersection property $\overline{M_{i_1 \cdots i_m}} =  \bigcap_{p=1}^m \overline{M_{i_p}} $ holds for all $m\leq n$. Then the sequence 
    $$      0 \rightarrow \Omega^q(\textbf{M}) \xrightarrow{ \ \Phi^* \ } \bigoplus_{i} \Omega^q(M_i) \xrightarrow{ \ \ \delta \ \ } \bigoplus_{i<j} \Omega^q\left(\overline{M_{ij}}\right) \xrightarrow{ \ \ \delta \ \ } \cdots \xrightarrow{ \ \ \delta \ \ } \Omega^q\left(\overline{M_{1\cdots n}}\right) \rightarrow 0       $$
    is exact for all $q$ in $\mathbb{N}$.
\end{theorem}       
\begin{proof}
    We will argue this by induction on the size of the indexing set $I$. Observe first that the binary case is already satisfied via the discussion in Section 3.1. For readability we will illustrate the inductive argument for $I=3$, since the general case is identical. According to our assumptions on the topology of $\textbf{M}$, we may view $\textbf{M}$ as an inductive colimit in which we glue $M_3$ to the union $M_1 \cup M_2$ along the subspace $M_{13}\cup M_{23}$. We denote by $\kappa$ the inclusion of the adjuction $M_1\cup_{f_{12}} M_2$ into $\textbf{M}$. Note that according to Section 1.1.1 of \cite{oconnell2023vector}, $\kappa$ is an open topological embedding. Diagrammatically, the differential forms on all the spaces in involved may be arranged as follows.

    \begin{center}
\adjustbox{scale=0.8}{
\begin{tikzcd}[row sep=4.8em, column sep = 1.4em]
            & 0 \arrow[d]                                                                 & 0 \arrow[d]                                                                                & 0 \arrow[d]                                                                                                    & 0 \arrow[d]                              &   \\
0 \arrow[r] & \ker{(\kappa^*)} \arrow[r, "\phi_3^*"] \arrow[d]                              & \Omega^q(M_3) \arrow[r, "{(-\overline{f_{13}}^*, -\overline{f_{23}}^*)}"] \arrow[d]                        & \Omega^q(\overline{M_{13}})\oplus \Omega^q(\overline{M_{23}}) \arrow[r, "\delta"] \arrow[d] & \Omega^q(\overline{M_{123}}) \arrow[r] \arrow[d] & 0 \\
0 \arrow[r] & \Omega^q(\textbf{M}) \arrow[r, "\Phi^*"] \arrow[d, "\kappa^*"']            & \bigoplus_{i} \Omega^q(M_i) \arrow[r, "\delta"] \arrow[d]                       & \bigoplus_{i<j} \Omega^q(\overline{M_{ij}}) \arrow[r, "\delta"] \arrow[d]                                      & \Omega^q(\overline{M_{123}}) \arrow[r] \arrow[d] & 0 \\
0 \arrow[r] & \Omega^q(M_{1} \cup M_{2}) \arrow[r] \arrow[d] & \Omega^q(M_{1})\oplus \Omega^q(M_{2}) \arrow[r] \arrow[d] & \Omega^q(\overline{M_{12}}) \arrow[r] \arrow[d]                                                                        & 0 \arrow[r] \arrow[d]                    & 0 \\
            & 0                                                                           & 0                                                                                          & 0                                                                                                              & 0                                        &  
\end{tikzcd}

}
\end{center}
In the above diagram, the objects in the first row are kernels of the vertical maps, and the horizontal maps are defined so as the make the diagram commute. Observe that the third row is an exact sequence. Thus, in order to argue for the exactness of the central row, it suffices to argue that the first row is exact. In principle we may argue for this directly. However, in order to illustrate the inductive nature of this argument, we will instead decompose the first row into the sequences: 
$$  0\rightarrow \ker(\kappa^*) \xrightarrow{\ \ \phi_2^* \ \ } \Omega^q(M_3) \xrightarrow{\ \ -\iota^* \ \ } \Omega^q(\overline{M_{13}}\cup \overline{M_{23}}) \rightarrow 0            $$
\vspace{-3mm}
$$  0 \rightarrow \Omega^q(\overline{M_{13}}\cup \overline{M_{23}}) \rightarrow \Omega^q(\overline{M_{13}}) \oplus \Omega^q(\overline{M_{23}}) \rightarrow \Omega^q(\overline{M_{123}})\rightarrow 0        $$
where here this splicing is formed by evaluating the kernel of the $\delta$ map. Regarding the former sequence, we will argue for exactness directly. Observe first that the map $\phi_3^*$ is manifestly injective, since any pair of distinct forms lying in $\ker{\kappa^*}$ will differ somewhere on $M_3$. Moreover, any form that lies in the image of the map $\phi_3^*$ is the restriction of a global form on $\textbf{M}$ that vanishes on $M_1 \cup M_2$. Since $\overline{M_{13}}\cup \overline{M_{23}}$ is a subset of $M_1 \cup M_2$, the restriction of such forms will also vanish. This ensures that $Im(\phi_3^*)\subseteq \ker{(-\iota^*)}$. For the converse, suppose that $\omega$ is some form on $M_3$ that maps to zero under $\iota^*$. This means that $\omega$ vanishes on $M_1$ and $M_2$. As such, the collection $(0,0,\omega)$ will induce a global form $\boldsymbol{\eta}$ on $\textbf{M}$ such that $\phi_3^*\boldsymbol{\eta}=\omega$ and $\boldsymbol{\eta} \in \ker{(-\iota^*)}$. Finally, since the set $\overline{M_{13}}\cup \overline{M_{23}}$ is closed, an extension by zero into $M_3$ will ensure the surjectivity of the restriction map $-\iota^*$. \\

In general, the latter sequence may not be exact. However, our assumptions on the intersections ensures that $\overline{M_{123}}=\overline{M_{13}} \cap \overline{M_{23}}$, which turns this sequence into a Hausdorff Mayer-Vietoris sequence for closed subsets. Exactness of this sequence then follows by an argument analogous to that of Lemma A.3 in the Appendix.\footnote{In the general inductive step, this sequence will expand into a longer sequence of the form of that found in Appendix A.2, and exactness will instead follow from that result.} Since the two aforementioned sequences are exact, and they splice to form the first row of the diagram, we may conclude that the first row of the diagram is exact, from which the result follows.
\end{proof}
According to the previous result, we may organise the de Rham cohomology for $\textbf{M}$ in terms of the bicomplex detailed in Figure \ref{FIG:CdR bicomplex for non-Hausdorff M}. This bicomplex relates the de Rham cohomology to the Čech cohomology of $\textbf{M}$ computed from the open cover $\{ M_i\}$. According to Theorem \ref{THM: MV long exact sequence for de rham} the rows of this bicomplex are exact, so we may use standard arguments (found in e.g. \cite[§8]{bott1982differential}) to conclude that the cohomology of this bicomplex coincides with the de Rham cohomology for $\textbf{M}$.\\

It is important to note that at this stage we have made no assumptions on the cohomology of each $M_i$. In particular, we did not assume that each $M_i$ (and their interesections) are contractible. This means that generally speaking, there is no guarantee that the columns of this bicomplex are exact, and thus we cannot prove a Čech-de Rham equivalence as in the Hausdorff case.  

\begin{figure}
    \centering
    \begin{tikzpicture}
    
    \draw[gray!25] (-3.5,1.25)--(9.5,1.25);
    \draw[gray!25] (-3.5,2.75)--(9.5,2.75);
    \draw[gray!25] (-3.5,4.25)--(9.5,4.25);
    \draw[gray!25] (-3.5,5.75)--(9.5,5.75);
    
    \draw[gray!25] (2.45,-1.5)--(2.45,7);
    \draw[gray!25] (5.45,-1.5)--(5.45,7);
    \draw[gray!25] (8.55,-1.5)--(8.55,7);

  \draw[thick] (-3.5,-0.25) -- (9.5,-0.25) node[below] {};
  \draw[thick] (-0.5,-1.5) -- (-0.5,7) node[left] {};

\node[] at (-3.05,0.5) {$0$};
\node[] at (-3.05,2) {$0$};
\node[] at (-3.05,3.5) {$0$};
\node[] at (-3.05,5) {$0$};

\node[] at (-1.5,0.5) {$\Omega^0(\textbf{M})$};
\node[] at (-1.5,2) {$\Omega^1(\textbf{M})$};
\node[] at (-1.5,3.5) {$\Omega^2(\textbf{M})$};
\node[] at (-1.5,5) {$\Omega^3(\textbf{M})$};
\node[] at (-1.5,6.5) {$\vdots$};

\node[] at (1,-1) {$ \check{C}^0(\textbf{M},\{M_i\}) $};
\node[] at (4,-1) {$ \check{C}^1(\textbf{M},\{M_i\}) $};
\node[] at (7,-1) {$ \check{C}^2(\textbf{M},\{M_i\}) $};

\node[] at (1,0.4) {$\bigoplus\limits_{i} \Omega^0(M_i)$};
\node[] at (1,1.9) {$\bigoplus\limits_{i} \Omega^1(M_i)$};
\node[] at (1,3.4) {$\bigoplus\limits_{i} \Omega^2(M_i)$};
\node[] at (1,4.9) {$\bigoplus\limits_{i} \Omega^3(M_i)$};
\node[] at (1,6.5) {$\vdots$};

\node[] at (4,0.4) {$\bigoplus\limits_{i<j} \Omega^0(\overline{M_{ij}})$};
\node[] at (4,1.9) {$\bigoplus\limits_{i<j} \Omega^1(\overline{M_{ij}})$};
\node[] at (4,3.4) {$\bigoplus\limits_{i<j} \Omega^2(\overline{M_{ij}})$};
\node[] at (4,4.9) {$\bigoplus\limits_{i<j} \Omega^3(\overline{M_{ij}})$};
\node[] at (4,6.5) {$\vdots$};

\node[] at (7,0.4) {$\bigoplus\limits_{i<j<k} \Omega^0(\overline{M_{ijk}})$};
\node[] at (7,1.9) {$\bigoplus\limits_{i<j<k} \Omega^1(\overline{M_{ijk}})$};
\node[] at (7,3.4) {$\bigoplus\limits_{i<j<k} \Omega^2(\overline{M_{ijk}})$};
\node[] at (7,4.9) {$\bigoplus\limits_{i<j<k} \Omega^3(\overline{M_{ijk}})$};
\node[] at (7,6.5) {$\vdots$};

\node[] at (9.35,5) {$\cdots$};
\node[] at (9.35,3.5) {$\cdots$};
\node[] at (9.35,2) {$\cdots$};
\node[] at (9.35,0.5) {$\cdots$};
\node[] at (9.35,-1) {$\cdots$};

\draw[->] (2.25,-1)--(2.75,-1);
\draw[->] (5.25,-1)--(5.75,-1);
\draw[->] (8.25,-1)--(8.75,-1);

\draw[->] (-2.85,0.5)--(-2.25,0.5);
\draw[->] (-0.85,0.5)--(0.05,0.5);
\draw[->] (1.95,0.5)--(2.95,0.5);
\draw[->] (5.05,0.5)--(5.95,0.5);
\draw[->] (8.3,0.5)--(8.95,0.5);

\draw[->] (-2.85,2)--(-2.15,2);
\draw[->] (-0.85,2)--(0.05,2);
\draw[->] (1.95,2)--(2.95,2);
\draw[->] (5.05,2)--(5.95,2);
\draw[->] (8.3,2)--(8.95,2);

\draw[->] (-2.85,3.5)--(-2.15,3.5);
\draw[->] (-0.85,3.5)--(0.05,3.5);
\draw[->] (1.95,3.5)--(2.95,3.5);
\draw[->] (5.05,3.5)--(5.95,3.5);
\draw[->] (8.3,3.5)--(8.95,3.5);

\draw[->] (-2.85,5)--(-2.15,5);
\draw[->] (-0.85,5)--(0.05,5);
\draw[->] (1.95,5)--(2.95,5);
\draw[->] (5.05,5)--(5.95,5);
\draw[->] (8.3,5)--(8.95,5);

\draw[->] (-1.5, 0.85)--(-1.5, 1.65);
\draw[->] (-1.5, 2.35)--(-1.5, 3.15);
\draw[->] (-1.5, 3.85)--(-1.5, 4.65);
\draw[->] (-1.5, 5.35)--(-1.5, 6.15);

\draw[->] (1, -0.65)--(1, 0.15);
\draw[->] (1, 0.85)--(1, 1.65);
\draw[->] (1, 2.35)--(1, 3.15);
\draw[->] (1, 3.85)--(1, 4.65);
\draw[->] (1, 5.35)--(1, 6.15);

\draw[->] (4, -0.65)--(4, 0.15);
\draw[->] (4, 0.85)--(4, 1.65);
\draw[->] (4, 2.35)--(4, 3.15);
\draw[->] (4, 3.85)--(4, 4.65);
\draw[->] (4, 5.35)--(4, 6.15);

\draw[->] (7, -0.65)--(7, 0.15);
\draw[->] (7, 0.85)--(7, 1.65);
\draw[->] (7, 2.35)--(7, 3.15);
\draw[->] (7, 3.85)--(7, 4.65);
\draw[->] (7, 5.35)--(7, 6.15);

    \end{tikzpicture}
    \caption{The Čech-de Rham Bicomplex for $\textbf{M}$.}
    \label{FIG:CdR bicomplex for non-Hausdorff M}
\end{figure}

\section{A Non-Hausdorff de Rham Theorem}

De Rham's theorem is a fundamental result that states an equivalence between de Rham cohomology and singular homology with real coefficients. In this section we will extend the result to non-Hausdorff manifolds, essentially by an appeal to the Mayer-Vietoris sequences of Section 3. \\

Throughout this section we will follow  \cite[Chpt. 18]{lee2013smooth}, and denote 
singular $n$-simplices by $\sigma: \Delta_q\rightarrow \textbf{M}$. We denote the space of singular $q$-chains by $C_q(\textbf{M})$, where here we view these as infinite-dimensional vector spaces over $\mathbb{R}$. An arbitrary singular $q$-chain is then a formal linear combination of continuous maps from $\Delta_q$ to $\textbf{M}$. We will denote the boundary operator for simplices by $\partial$, defined as per usual. \\

Ordinarily, the Hausdorff de Rham theorem is proved by defining an ``integration over chains", which relates differential forms to singular chains. However, there is a subtely here: in order to integrate over chains, differential forms need to be pulled back to $\Delta_q$ along \textit{smooth} maps. As such, integration over chains actually induces a pairing between de Rham cohomology and the ``smooth singular homology", which \textit{a priori} may differ from singular homology. This issue may be resolved by arguing that smooth singular homology is equivalent to singular homology via some homotopic approximation of continuous maps by smooth ones. In this section we will follow a similar approach, with some key modifications. 

\subsection{Smooth Singular Homology}
We start our derivation of de Rham's theorem with a confirmation that smooth singular homology coincides with singular homology in the non-Hausdorff regime. In this context, ``smooth" means that any simplex $\sigma: \Delta_q \rightarrow \textbf{M}$ has a smooth extension in some neighbourhood of each point in $\Delta_q$. By standard arguments it can be shown that the boundary of a smooth chain is again smooth, and thus the space $C^\infty_q(\textbf{M})$ of smooth $q$-chains forms a subcomplex of $C_q(\textbf{M})$, with the inclusion $i: C^\infty_q(\textbf{M})\rightarrow C_q(\textbf{M})$ being a chain map. \\

For Hausdorff manifolds, the equivalence between the smooth and singular homologies is typically argued by appealing to Whitney's Embedding theorem, a consequence of which is that any continuous map between manifolds is homotopic to a smooth one. This general observation can be used to construct a homotopy equivalence between $C_q(\textbf{M})$ and $C_q^\infty(\textbf{M})$ (cf. \cite[Thm 18.7]{lee2013smooth}). Since Hausdorffness is a hereditary property, there are no circumstances under which a non-Hausdorff manifold may be embedded into Euclidean space, of any dimension. Moreover, it is unclear whether all continuous maps between non-Hausdorff manifolds can be approximated by smooth maps. Instead of addressing the problem head-on, we will instead prove the equivalence between $H^\infty_p(\textbf{M})$
and $H_p(\textbf{M})$ by alternate means. \\

Since the standard Hausdorff argument is not available to us, instead we will derive a Mayer-Vietoris sequence for smooth singular homology from first principles. For a binary adjunction, this would amount to a proof of the exactness of the sequence $$  0 \rightarrow C^\infty_q(M_{12}) \xrightarrow{ \ ((\iota_{12})_* , -(f_{12})_*)  \ } C^\infty_q(M_1)\oplus C^\infty_q(M_2) \xrightarrow{ \ (\phi_1)_* + (\phi_2)_*  \ } C^\infty_q(\textbf{M}) \rightarrow 0 $$
where the maps involved are analogues of the pushforwards found in \cite{lee2013smooth}. As a matter of fact, the above sequence is trivially exact if we replace the last non-zero term with the space $C^\infty_q(M_1 + M_2)$, which is the space of smooth singular chains in $\textbf{M}$ that take image in either $M_1$ or $M_2$. This is a subspace of $C^\infty_q(\textbf{M})$, and thus there is an inclusion $\iota:C^\infty_q(M_1 + M_2) \rightarrow C^\infty_q(\textbf{M})$. In the singular case, there is a well-known construction which guarantees that the inclusion $\iota$ is a homotopy equivalence \cite{hatcher2002algebraic}[Prop. 2.21]. As noted in \cite[Pg. 135]{felix2012rational}, the same argument will hold for smooth singular chains as well. With an eye towards Theorem \ref{THM: Equivalence between smooth singular and singular}, we will state the result in slightly more general language.
\begin{lemma}\label{LEM: homotopy equivalence between smooth singular and singular chains}
    Let $\mathcal{U}=\{U_1, U_2\}$ be an open cover of $\textbf{M}$, and denote by $C^\infty_q(U_1 + U_2)$ the space of smooth singular chains $c=\sum_\alpha c_\alpha \sigma_\alpha$ such that each $\sigma_\alpha$ has image contained entirely in either $U_1$ or $U_2$. Then the inclusion $\iota: C^\infty_q(U_1 + U_2) \rightarrow C^\infty_q(\textbf{M}) $ is a homotopy equivalence.
\end{lemma}
\noindent We may apply the Snake Lemma to the short exact sequences
$$  0 \rightarrow C^\infty_q(M_1 \cap M_2) \xrightarrow{ \ ((\iota_{12})_* , -(f_{12})_*)  \ } C^\infty_q(M_1)\oplus C^\infty_q(M_2) \xrightarrow{ \ (\phi_1)_* + (\phi_2)_*  \ } C^\infty_q(M_1 + M_2) \rightarrow 0 $$
and then use Lemma \ref{LEM: homotopy equivalence between smooth singular and singular chains} to replace each $H^\infty_p(M_1 + M_2)$ with $H^\infty_p(\textbf{M})$ in the resulting Mayer-Vietoris sequence. We will now use this observation to prove an equivalence between the groups $H^\infty_p(\textbf{M})$ and $H_p(\textbf{M})$. 

\begin{theorem}\label{THM: Equivalence between smooth singular and singular}
    Let $\textbf{M}$ be a non-Hausdorff manifold satisfying the properties of Remark \ref{REM: assumptions in this paper}. Then the smooth singular homology group $H^\infty_p(\textbf{M})$ is isomorphic to the singular homology group $H^p(\textbf{M})$ for all $p$ in $\mathbb{N}$.
\end{theorem}
\begin{proof}
    We proceed by induction on the size of the indexing set $I$ defining $\textbf{M}$. Suppose first that $\textbf{M}$ is a binary colimit. By the discussion preceeding this theorem, there is a Mayer-Vietoris long exact sequence for smooth singular homology. The standard derivation of the Mayer-Vietoris sequence for singular homology (cf. \cite[§2.2]{hatcher2002algebraic}) also holds in this setting. The inclusions $i: C^\infty_q( \cdot) \rightarrow C_q(\cdot )$ are all homomorphisms that commute with the relevant boundary maps, as depicted below. 

    \begin{center}
\begin{tikzcd}[row sep=3.5em, column sep=1.6em]
0 \arrow[r] & C^\infty_q(M_{12}) \arrow[r] \arrow[d, "i"'] & C^\infty_q(M_{1})\oplus C^\infty_q(M_{2}) \arrow[r] \arrow[d, "i"'] & C^\infty_q(M_{1}+M_{2}) \arrow[r] \arrow[d, "i"'] & 0 \\
0 \arrow[r] & C_q(M_{12}) \arrow[r]                        & C_q(M_{1})\oplus C_q(M_{2}) \arrow[r]                               & C_q(M_{1}+M_{2}) \arrow[r]                        & 0
\end{tikzcd}
    \end{center}
\noindent As such, we may use the naturality of the Snake Lemma to construct the following commutative diagram. \vspace{-3mm}
\begin{center}        
\adjustbox{scale=0.74}{\begin{tikzcd}[row sep=4.2em, column sep=2em]
\cdots \arrow[r] & H^\infty_q(M_{12}) \arrow[d, "i_*"'] \arrow[r] & H^\infty_q(M_1)\oplus H^\infty_q(M_2) \arrow[d, "i_*"'] \arrow[r] & H^\infty_{q}(M_1 + M_2) \arrow[d, "i_*"'] \arrow[r] & H^\infty_{q-1}(M_{12}) \arrow[d, "i_*"'] \arrow[r] & H^\infty_{q-1}(M_1)\oplus H^\infty_{q-1}(M_2) \arrow[d, "i_*"'] \arrow[r] & \cdots \\
\cdots \arrow[r] & H_q(M_{12}) \arrow[r]                         & H_q(M_1)\oplus H_q(M_2) \arrow[r]                                & H_{q}(M_1 + M_2) \arrow[r]                         & H_{q-1}(M_{12}) \arrow[r]                         & H_{q-1}(M_1)\oplus H_{q-1}(M_2) \arrow[r]                                & \cdots
\end{tikzcd}}
    \end{center}

 Theorem 18.7 of \cite{lee2013smooth} ensures that the inclusion maps between Hausdorff components of this diagram are isomorphisms. Since the diagram commutes, we may apply the Five Lemma to conclude that the central $i_*$-map depicted above is also an isomorphism. By applying Lemma \ref{LEM: homotopy equivalence between smooth singular and singular chains} and the analogue for singular chains, it follows that $$  H^\infty_{q}(\textbf{M}) \cong H^\infty_{q}(M_1 + M_2) \cong H_{q}(M_1 + M_2) \cong H_{q}(\textbf{M}),  $$
    as required. \\

    For the inductive case, suppose that smooth singular homology coincides with singular homology for all non-Hausdorff manifolds of size $n$. Suppose now that we have a non-Hausdorff manifold built as a colimit of $(n+1)$-many Hausdorff manifolds $M_i$. We may use the inductive representation of $\textbf{M}$ and rewrite it as a binary adjunction space $\textbf{M} \cong \textbf{N} \cup_F M_{n+1}$, where $\textbf{N}$ is the colimit of the first $n$-many spaces in $\textbf{M}$, and $\textbf{N}$ is glued to $M_{n+1}$ along the subspace $A:= \bigcup_{i\leq n} M_{i(n+1)}$.
    Since the subsets $\textbf{N}$ and $M_{n+1}$ form an open cover of $\textbf{M}$, we may apply the result of Lemma \ref{LEM: homotopy equivalence between smooth singular and singular chains} to conclude that 
    $$  0 \rightarrow C^\infty_q(A) \xrightarrow{ \ ((\iota_{A})_* , -F_*)  \ } C^\infty_q(\textbf{N})\oplus C^\infty_q(M_{n+1}) \xrightarrow{ \ (\varphi_1)_* + (\varphi_2)_*  \ } C^\infty_q(\textbf{N}+M_{n+1}) \rightarrow 0 $$
    is a short exact sequence, where here the maps $\varphi_i$ are the canonical embeddings of $\textbf{N}$ and $M_{n+1}$ into $\textbf{M}$. 
    The inclusions $i$ into singular cochains will still commute with boundary maps, and thus the descended map $i_*$ will still be a chain map for the homologies. Since $A$ and $M_{n+1}$ are Hausdorff manifolds, the $i_*$ map on their cochains will be isomorphisms by \cite{lee2013smooth}. Moreover, the inductive hypothesis tells us that the map $i_*: H^\infty_q(\textbf{N})\rightarrow H_q(\textbf{N})$ is an isomorphism for all $q$. As such, we may again apply the Five Lemma to conclude that $ H^\infty_{q}(\textbf{N} + M_{n+1}) \cong H_{q}(\textbf{N} + M_{n+1})$. The result then follows as an application of Lemma \ref{LEM: homotopy equivalence between smooth singular and singular chains}.   
\end{proof}

\subsection{Integration Over Chains}
With Theorem \ref{THM: Equivalence between smooth singular and singular} in mind, we will now restrict our attention to the smooth singular homology of $\textbf{M}$. As with the Hausdorff case, we may define a notion of \textit{integration over chains}, defined by pulling back differential forms on $\textbf{M}$ to the simplex $\Delta_q$, and then evaluating the integral there. Given a smooth singular chain $c = \sum_{\alpha =1}^m c_\alpha \sigma_\alpha$ and a differential form $\boldsymbol{\omega}$ on $\textbf{M}$, we define the integral of $\boldsymbol{\omega}$ over $c$ as: 
$$  \mathcal{I}(\boldsymbol{\omega}, c) := \sum_{\alpha = 1}^m c_\alpha \int_{\Delta_q} \sigma_\alpha^* \boldsymbol{\omega}. $$

In Section 2, we saw that Stoke's theorem will generally not hold when integrating forms on $\textbf{M}$. However, integration over chains circumvents this issue by evaluating the integral over the simplex $\Delta_q$, interpreted as a submanifold of $\mathbb{R}^q$. In particular, the Hausdorff version of Stoke's theorem applies (cf \cite{lee2013smooth}[Thm. 18.12]), and we may conclude the following. 
\begin{lemma}\label{LEM: stokes theorem fails}
    If $c$ is a smooth singular $q$-chain and $\boldsymbol{\omega}$ is a $(p-1)$-form on $\textbf{M}$, then  $$  \int_c d\boldsymbol{\omega} = \int_{\partial c} \boldsymbol{\omega}. $$  
\end{lemma}

As with the Hausdorff case, it can be shown that the map $\mathcal{I}$ descends to cohomology. We will hereafter interpret integration as a map $$ \mathcal{I}: H^p(\textbf{M})\rightarrow H^p_{\infty}(\textbf{M};\mathbb{R}),$$
where the image is the smooth singular cohomology, interpreted as the canonical dual space of $H^\infty_p(\textbf{M})$. 

\subsection{The de Rham Theorem}

We will now combine the results of Section 3 with our discussion of smooth singular homology to construct a non-Hausdorff version of de Rham's theorem. The full argument is detailed below. 

\begin{theorem}\label{THM: de Rhams theorem}
    Suppose that $\textbf{M}$ is a non-Hausdorff manifold satisfying the criteria of \ref{REM: assumptions in this paper} and Theorem \ref{THM: MV long exact sequence for de rham}, and suppose furthermore that all unions of the sets $M_{ij}$ are regular-open sets. Then the integral map $\mathcal{I}: H_{dR}^p(\textbf{M})\rightarrow H^p_\infty(\textbf{M};\mathbb{R})$ is an isomorphism. 
\end{theorem}
\begin{proof}
Our argument will be similar in structure to that of Theorem \ref{THM: Equivalence between smooth singular and singular}, that is, we will proceed by induction on the size of the indexing set $I$. Suppose first that $\textbf{M}$ is a colimit of two Hausdorff manifolds $M_1$ and $M_2$. According to our discussion in Section 3, we may construct a Mayer-Vietoris sequence for the de Rham cohomology of $\textbf{M}$ by using the cohomologies of the $M_i$, and the region $M_{12}$. Following \cite{lee2013smooth}[§18.4], we note that the descended integration maps $\mathcal{I}$ commute with smooth maps and the connecting homomorphisms of de Rham and smooth singular cohomologies. As such, we have the following  commutative diagram. 
\begin{center}        
\adjustbox{scale=0.71}{\begin{tikzcd}[row sep=4.6em, column sep=1.5em]
\cdots \arrow[r] & H_{dR}^{q-1}(M_1)\oplus H_{dR}^{q-1}(M_2) \arrow[d, "\mathcal{I}"'] \arrow[r] &  H_{dR}^{q-1}(M_{12}) \arrow[d, "\mathcal{I}"'] \arrow[r] & H_{dR}^{q}(\textbf{M}) \arrow[d, "\mathcal{I}"'] \arrow[r] & H_{dR}^{q}(M_1)\oplus H_{dR}^{q}(M_2) \arrow[d, "\mathcal{I}"'] \arrow[r] & H_{dR}^{q}(M_{12})  \arrow[d, "\mathcal{I}"'] \arrow[r] & \cdots \\
\cdots \arrow[r] & H^{q-1}_{\infty}(M_1;\mathbb{R})\oplus H^{q-1}_{\infty}(M_2;\mathbb{R})  \arrow[r]                         &  H^{q-1}_{\infty}(M_{12};\mathbb{R}) \arrow[r]             & H^{q}_{\infty}(\textbf{M};\mathbb{R}) \arrow[r]                         & H^{q}_{\infty}(M_1;\mathbb{R})\oplus H^{q}_{\infty}(M_2;\mathbb{R}) \arrow[r]                         & H^{q}_{\infty}(M_{12};\mathbb{R})  \arrow[r]                                & \cdots
\end{tikzcd}}
    \end{center}
We may then apply the Hausdorff version of de Rham's theorem to all of the Hausdorff columns of this commutative diagram. It then follows from the Five Lemma that the map $\mathcal{I}:H_{dR}^{q}(\textbf{M}) \rightarrow H^{q}_{\infty}(\textbf{M};\mathbb{R})$ is an isomorphism.\\

Suppose now that the map $\mathcal{I}$ is an isomorphism for all non-Hausdorff manifolds constructed as a colimit of $n$-many Hausdorff manifolds $M_i$. Let $\textbf{M}$ be a non-Hausdorff manifold with indexing set of size $(n+1)$. Similarly to Theorem \ref{THM: Equivalence between smooth singular and singular}, we will consider the inductive colimit $\textbf{M} \cong \textbf{N}\cup_F M_{n+1}$. As with the binary case, we again have a Mayer-Vietoris sequence for the de Rham and smooth singular cohomologies of $\textbf{M}$, this time using the open cover $\{\textbf{N}, M_{n+1} \}$.\footnote{Technically, the Mayer-Vietoris sequence that we will obtain will have terms of the form $H^q(\overline{A})$, however, due to our assumption that $A$ is regular-open, we may replace such terms with $H^q(\overline{A})$.} Again by Theorem 18.12 of \cite{lee2013smooth}, the descended integral map $\mathcal{I}$ commutes with pullbacks and connecting homomorphisms, so we may proceed as in the binary case and interpret $\mathcal{I}$ as a chain map between the two Mayer-Vietoris sequences. The result then follows from an application of the induction hypothesis applied to $\textbf{N}$ and $A$, together with the Five Lemma. 
\end{proof}

Using the above argument, together with Theorem \ref{THM: Equivalence between smooth singular and singular}, we obtain the non-Hausdorff de Rham's theorem as an immediate corollary. 
 
\subsubsection*{The Line with Two Origins}
In the de Rham theorem of \ref{THM: de Rhams theorem} we assumed that our non-Hausdorff manifolds were formed according to the assumptions of \ref{THM: MV long exact sequence for de rham}. We will now illustrate the necessity of these assumptions with an example. Consider the line with two origins, constructed from two copies of $\mathbb{R}$ glued along the open subset $(-\infty, 0)\cup(0,\infty)$ as in Figure \ref{fig: line with two origins}. Observe that the open set $M_{12}$ is not regular open in $\mathbb{R}$. In particular, this means that a Mayer-Vietoris sequence for the de Rham cohomology will take the form $$ 0 \rightarrow \mathbb{R} \rightarrow \mathbb{R}\oplus\mathbb{R} \rightarrow \mathbb{R} \rightarrow H_{dR}^1(\textbf{M}) \rightarrow 0 \rightarrow 0  $$
where here in the third term we have used that $\overline{M_{12}}=\mathbb{R}$. Exactness of this sequence ensures that the $H_{dR}^1(\textbf{M})=0$. \\

On the other hand, we may compute the singular homology for this space by alternate means. By inspection it should be clear that the first singular homology of $\textbf{M}$ should carry some extra term, since $H^S_1(\textbf{M})$ is the abelianization of the first fundamental group of $\textbf{M}$, and here there are non-contractible loops that wrap around the two copies of the origin. Alternatively, we may use a Mayer-Vietoris sequence for singular homology and then dualize it to express the group $H^1_S(\textbf{M},\mathbb{R})$.  We do not need to use the boundary of $M_{12}$ in this sequence, so the simplicial cohomology of $\textbf{M}$ may be described with the following long exact sequence 
$$  0\rightarrow \mathbb{R} \rightarrow \mathbb{R}\oplus\mathbb{R} \rightarrow \mathbb{R}\oplus \mathbb{R} \rightarrow H^1_S(\textbf{M}, \mathbb{R}) \rightarrow 0 \rightarrow 0        $$
where we observe an extra copy of $\mathbb{R}$ in the third term, due to the two connected components of $M_{12}$. By exactness, it follows that $H^1_S(\textbf{M},\mathbb{R}) = \mathbb{R}$, which differs from the associated de Rham cohomology. \\

Interestingly, the line with two origins may also serve as a counterexample for the homotopy invariance of non-Hausdorff de Rham cohomology. Indeed, we may consider a manifold $\textbf{N}$ in which the two copies of the real line are glued along the sets $B=(-\infty, -1)\cup (1,\infty)$. This will yield a non-Hausdorff manifold that is homotopy equivalent to $\textbf{M}$, yet is constructed by gluing along regular open sets. As such, the manifold $\textbf{N}$ falls under the scope of the non-Hausdorff de Rham theorem, from which we may deduce that $H_{dR}^1(\textbf{N})=\mathbb{R}\neq 0 = H_{dR}^1(\textbf{M})$. \\

\section{The Gauss-Bonnet Theorem}

For a compact Hausdorff surface $\Sigma$ with boundary, the Gauss-Bonnet theorem relates the total scalar curvature to the Euler characteristic via the equality $$  2\pi \chi(\Sigma) = \frac{1}{2}\int_\Sigma R dA + \int_{\partial \Sigma} \kappa d\gamma,        $$
where here the latter term computes the geodesic curvature of the boundary of $\Sigma$. In this section we will prove a non-Hausdorff version of this statement. In order to emphasise that we are working with surfaces, throughout this section we will denote our non-Hausdorff manifold by $\boldsymbol{\Sigma}$, with its Hausdorff submanifolds denoted similarly. We will see that in a manner similar to that of Theorem \ref{LEM: stokes theorem fails}, even if the manifold $\boldsymbol{\Sigma}$ has no boundary, the total curvature of $\boldsymbol{\Sigma}$ will take additional contributions from the internal Hausdorff-violating submanifolds. In order to properly understand this theorem, we will start with a brief discussion of curvature in non-Hausdorff manifolds. 

\subsection{Scalar Curvature for Non-Hausdorff Manifolds}
According to Theorem \ref{THM: sections are fibred product} the sections of the tensor bundle $T^{(2,0)}\boldsymbol{\Sigma}$ are a fibre product of the spaces $\Gamma(T^{(2,0)}\Sigma_i)$. As such, we may describe a Riemmanian metric on $\boldsymbol{\Sigma}$ using a collection of metrics $g_i$ defined on each $\Sigma_i$, such that the isometry condition $g_i = f_{ij}^* \circ g_j$ holds on $\Sigma_{ij}$ for all $i,j$ in $I$. Explicitly, we have the following result, paraphrased from \cite{oconnell2023vector}. 

\begin{lemma}\label{LEM: Riemannian non-hausdorff manifold}
    Let $\boldsymbol{\Sigma}$ be a non-Hausdorff manifolds built according to \ref{REM: assumptions in this paper}. If each $\Sigma_i$ is equipped with a Riemannian metric $g_i$ such that $f_{ij}:\Sigma_{ij}\rightarrow M_j$ is an isometric embedding, then $\boldsymbol{\Sigma}$ admits a Riemannian metric $\textbf{g}$ such that the canonical maps $\phi_i: \Sigma_i \rightarrow \boldsymbol{\Sigma}$ are open isometric embeddings. 
\end{lemma}
In coordinate-free notation, a connection on $\boldsymbol{\Sigma}$ may be defined as an operator $$ \nabla: \mathfrak{X}(\boldsymbol{\Sigma}) \times \mathfrak{X}(\boldsymbol{\Sigma}) \rightarrow \mathfrak{X}(\boldsymbol{\Sigma}), \ (\textbf{X},\textbf{Y})\mapsto \nabla_\textbf{X} \textbf{Y}  $$
that is $C^\infty(\boldsymbol{\Sigma})$-linear in the first argument, and satisfies the Liebniz property $$ \nabla_\textbf{X}(\textbf{r}\textbf{Y}) = \textbf{r}\nabla_\textbf{X}(\textbf{Y}) + \textbf{X}(\textbf{r})\textbf{Y}  $$
in the second argument. Given a metric $\textbf{g}$ on $\boldsymbol{\Sigma}$, we may define the Levi-Civita connection as: $$ \boldsymbol{\nabla}(\textbf{X},\textbf{Y})([x,i]) := \nabla^i(X_i, Y_i)(x),$$
where $\nabla^i$ is the Levi-Civita connection for each $(\Sigma_i, g_i)$. Note that the isometries $\overline{f_{ij}}$ together with the uniqueness of Hausdorff Levi-Civita connections ensure that the above definition is well-defined and unique. The following result summarises some relevant properties of $\boldsymbol{\nabla}$.

\begin{lemma}
    Let $\boldsymbol{\nabla}$ be the connection defined above. Then 
    \begin{enumerate}
        \item $\boldsymbol{\nabla}$ is a local operator
        \item $\boldsymbol{\nabla}$ is compatible with the metric $\textbf{g}$
        \item $\boldsymbol{\nabla}$ is torsion-free. 
    \end{enumerate}
\end{lemma}
\begin{proof}
    Semi-locality of $\boldsymbol{\nabla}$ is established by construction, and locality follows from an argument similar to the exterior derivative (cf Theorem \ref{THM:exterior derivative properties}). For readability we will illustrate the latter two properties for the case that $I=2$, since the general case simply requires more indices. For compatibility with the metric, we have 
    \begin{align*}
        \textbf{X}(\textbf{g}(\textbf{Y},\textbf{Z})) & = \big(X_1(g_1(Y_1,Z_1)), X_2(g_2(Y_2,Z_2)) \big) \\
        & = \left( g_1(\nabla^1_{X_1}Y_1,Z_1) + g_1(Y_1,\nabla^1_{X_1}Z_1) , g_2(\nabla^2_{X_2}Y_2,Z_2) + g_2(Y_2,\nabla^2_{X_2}Z_2) \right) \\
        & = \left( g_1(\nabla^1_{X_1}Y_1,Z_1), g_2(\nabla^2_{X_2}Y_2,Z_2) \right)   + \left(g_1(Y_1,\nabla^1_{X_1}Z_1), g_2(Y_2,\nabla^2_{X_2}Z_2) \right) \\
        & =  \textbf{g}(\boldsymbol{\nabla}_{\textbf{X}}\textbf{Y},\textbf{Z}) +\textbf{g}(\textbf{Y},\boldsymbol{\nabla}_{\textbf{X}}\textbf{Z}),
    \end{align*}
    where in the second line we have used the metric compatibility of the Hausdorff Levi-Civita connections. To see that $\boldsymbol{\nabla}$ is torsion-free, we may use the expression for the Lie bracket $$ [\textbf{X},\textbf{Y}] = \left([X_1,Y_1], [X_2,Y_2] \right)     $$ derived in Section 2.1. We may then use the Hausdorff Levi-Civita connections to deduce that\begin{align*}
        \nabla_\textbf{X} \textbf{Y} - \nabla_\textbf{Y} \textbf{X} & = (\nabla^1_{X_1} Y_1 - \nabla^1_{Y_1} X_1,\nabla^2_{X_2} Y_2 - \nabla^2_{Y_2} X_2 ) \\
        & = \left( [X_1,Y_1], [X_2,Y_2] \right) \\
        & = [\textbf{X},\textbf{Y}], 
    \end{align*}
    as required. 
\end{proof}
Now that we have established a Levi-Civita connection for $\boldsymbol{\Sigma}$, we may define the Riemann curvature tensor as in the Hausdorff case. Given the expressions for $\boldsymbol{\nabla}$ and the Lie derivative on $\boldsymbol{\Sigma}$, we see that the Riemann curvature tensor for $\boldsymbol{\Sigma}$ will again be a fibred product of the same tensors defined on each $\Sigma_i$. Again, this will be a local operator since both $\boldsymbol{\nabla}$ and the Lie derivative are. Since the maps $f_{ij}$ are isometries and the connection $\boldsymbol{\nabla}$ is local, we may compute the Ricci tensor, and consequently the Ricci scalar $\textbf{R}$  on $\boldsymbol{\Sigma}$ will be locally equivalent to the same value computed within each Hausdorff manifold $\Sigma_i$.  \\

By a similar line of reasoning, the volume form of $\boldsymbol{\Sigma}$ can be described according to Lemma \ref{LEM: orientable manifold}. We may thus write the total scalar curvature of $\boldsymbol{\Sigma}$ as 
$$ \int_{\boldsymbol{\Sigma}} \textbf{R} dA = \int_{\boldsymbol{\Sigma}} \textbf{R} \sqrt{\det{(g)}} d^2x.$$
Using Theorem \ref{THM: Inductive form of integral}, we may always decompose this integral into a sub-additive collection of scalar curvatures defined on each $\Sigma_i$ and their various intersections.

\subsection{The Euler Characteristic}
Before getting to a proof of the Gauss-Bonnet theorem, we will first describe the Euler characteristic of the manifold $\boldsymbol{\Sigma}$. Since simplicial homology is not available for non-Hausdorff manifolds, we will instead define the Euler characteristic as $$\chi(\boldsymbol{\Sigma}) := \sum_{q=0}^d (-1)^q  \ b_q(\boldsymbol{\Sigma}),$$ where we define these Betti numbers via singular homology groups. According to this definition, the Euler characteristic will satisfy the following sub-additive property.

\begin{theorem}\label{THM: Euler characteristic}
The Euler characteristic of $\boldsymbol{\Sigma}$ equals $$\chi(\boldsymbol{\Sigma}) = \sum_{p\in I} \sum_{i_1,...,i_p \in I} (-1)^{p+1}\chi(\Sigma_{i_1,...,i_p}).$$ 
\end{theorem}
\begin{proof}
We argue by induction on the size of the indexing set $I$. Suppose first that $I=2$, that is, that $\boldsymbol{\Sigma}$ is a binary colimit. According to the results of Section 4.1, we may construct a long exact sequence that relates the spaces $H^q(\boldsymbol{\Sigma})$ to the cohomologies of $\Sigma_1$, $\Sigma_2$ and $\Sigma_{12}$. Since this sequence is a long exact sequence of vector spaces, the alternating sum of the dimensions of the entries always equals zero. Spelling this out, we have: 
\begin{align*}
0 & = \sum_q (-1)^q \big(b_q(\boldsymbol{\Sigma}) - b_q(\Sigma_1) - b_q(\Sigma_2) + b_q(\Sigma_{12})\big) \\
& = \sum_q (-1)^q b_q(\boldsymbol{\Sigma}) - \sum_q (-1)^q\big(b_q(\Sigma_1) + b_q(\Sigma_2)\big) + \sum_q (-1)^q b_q(\Sigma_{12}) \\
& = \chi(\boldsymbol{\Sigma}) - \chi(\Sigma_1) - \chi(\Sigma_2) + \chi(\Sigma_{12}) 
\end{align*}
from which the equality follows. \\

Suppose now that the hypothesis holds for all colimits of size $I = n$, and let $\boldsymbol{\Sigma}$ have indexing set size $n+1$. As with the proofs of Section 4, we may view $\boldsymbol{\Sigma}$ as an inductive colimit $\tilde{\boldsymbol{\Sigma}} \cup_{F} \Sigma_{n+1}$. The cohomologies of these spaces can be arranged into the following Mayer-Vietoris sequence 
\begin{center}
\adjustbox{scale=0.9}{\begin{tikzcd}[arrow style=math font,cells={nodes={text height=2ex,text depth=0.75ex}}]
      
       \cdots \arrow[r] & H^q (\boldsymbol{\Sigma}) \arrow[r, "\varphi^*"] & H^q(\tilde{\boldsymbol{\Sigma}}) \oplus  H^q(\Sigma_{n+1}) \arrow[r,"\iota_A^* - F^*"] \arrow[draw=none]{d}[name=Y, shape=coordinate]{} & H^q(A) \arrow[curarrow=Y]{dll}{} & \\
       & H^{q-1} (\boldsymbol{\Sigma}) \arrow[r, "\varphi^*"] & H^{q-1}(\tilde{\boldsymbol{\Sigma}}) \oplus  H^{q-1}(\Sigma_{n+1}) \arrow[r, "\iota_A^* - F^*"]  & H^{q-1}(A) \arrow[r] & \cdots 
\end{tikzcd}}
\end{center}
where here $A$ is the union of the subspaces $\Sigma_{i n+1}$. Since this a long exact sequence of vector spaces, we may again use that the alternating sum of dimensions vanish to deduce that: 
\begin{align*}
0 & = \sum_q (-1)^q \big(b_q(\boldsymbol{\Sigma}) - b_q(\tilde{\boldsymbol{\Sigma}}) - b_q(\Sigma_{n+1}) + b_q(A)|\big) \\
& = \sum_q (-1)^q b_q(\boldsymbol{\Sigma}) - \sum_q (-1)^q\big(b_q(\tilde{\boldsymbol{\Sigma}}) + b_q(\Sigma_{n+1})\big) + \sum_q (-1)^q b_q(A) \\
& = \chi(\boldsymbol{\Sigma}) - \chi(\tilde{\boldsymbol{\Sigma}}) - \chi(\Sigma_{n+1}) + \chi(A). 
\end{align*}
We may then apply the inductive hypothesis to both $\tilde{\boldsymbol{\Sigma}}$ and $A=\bigcup_{i} \Sigma_{in+1}$ to yield the following
\begin{align*}
    \chi(\boldsymbol{\Sigma}) & = \chi(\tilde{\boldsymbol{\Sigma}}) + \chi(\Sigma_{n+1}) - \chi(A) \\ 
    & = \left(\sum_{k\leq n} \sum_{i_1,...,i_k \in I} (-1)^{k+1}\chi(\Sigma_{i_1,...,i_k})\right) + \chi(\Sigma_{n+1}) \\
    & \ \ \ \ \ \ \ \ - \left(  \sum_{k\leq n} \sum_{i_1,...,i_k \in I} (-1)^{k+1}\chi(\Sigma_{i_1,...,i_k, n+1})  \right) \\
    & = \sum_{k\leq n+1} \sum_{i_1,...,i_k \in I} (-1)^{k+1}\chi(\Sigma_{i_1,...,i_k}),
\end{align*}
 as required. 
\end{proof}

\subsection{Proof of the Theorem}
We will prove the Gauss-Bonnet theorem in a similar manner to the non-Hausdorff Stoke's theorem of Section 2. We start with the binary version of the theorem. 

\begin{lemma}
    Suppose that $\boldsymbol{\Sigma}$ is a non-Hausdorff manifold that can be constructed as a binary colimit of Hausdorff manifolds $\Sigma_1$ and $\Sigma_2$ according to \ref{REM: assumptions in this paper} and \ref{THM: MV long exact sequence for de rham}. If $\Sigma_1$ and $\Sigma_2$ are compact and without boundary, then:  
    $$ 2\pi \chi(\boldsymbol{\Sigma}) =  \frac{1}{2}\int_{\boldsymbol{\Sigma}} \textbf{R}dA - \int_{\partial \Sigma_{12}} \kappa d\gamma. $$
\end{lemma}
\begin{proof}
According to our discussion of curvature in Section 5.1, we may decompose the total scalar curvature of $\boldsymbol{\Sigma}$ into the scalar curvatures of the $\Sigma_i$ by using Theorem \ref{THM: subadditivity for integration}. Using the Hausdorff Gauss-Bonnet theorem where possible, we have the following string of equivalences:
    \begin{align*}
        2\pi \chi(\boldsymbol{\Sigma}) & = 2\pi \left( \chi(\Sigma_1) + \chi(\Sigma_2) - \chi(\overline{\Sigma_{12}}) \right) \\
        & = \frac{1}{2}\int_{\Sigma_1} R dA + \frac{1}{2}\int_{\Sigma_2} R dA -  \frac{1}{2}\int_{\overline{\Sigma_{12}}} R dA - \int_{\partial \Sigma_{12}} \kappa d\gamma \\
        & = \int_{\boldsymbol{\Sigma}} R dA - \int_{\partial \Sigma_{12}} \kappa d\gamma, 
    \end{align*}
    where here we have used the fact that $\overline{\Sigma_{12}}$ is a manifold with boundary. 
\end{proof}
By an inductive argument, we may obtain the following general version of the Gauss-Bonnet theorem for non-Hausdorff manifolds. Again for simplicity we restrict our attention to a colimit in which the $\Sigma_i$ have no boundary. 

\begin{theorem}\label{THM: Gauss-Bonnet general}
     Suppose that $\boldsymbol{\Sigma}$ is a non-Hausdorff manifold that can be constructed as a colimit of $n$-many Hausdorff manifolds $\Sigma_i$ according to Remark \ref{REM: assumptions in this paper}. If each $\Sigma_i$ is compact and without boundary, and each $\Sigma_{ij}$ and their unions are regular-open sets, then  
    $$ 2\pi \chi(\boldsymbol{\Sigma}) =  \frac{1}{2}\int_{\boldsymbol{\Sigma}} \textbf{R}dA + \sum_{p=2}^n (-1)^{p+1} \sum_{i_1, \cdots , i_p \in I} \left(\int_{\partial M_{i_1 \cdots i_p}} \kappa d\gamma\right). $$ 
\end{theorem}
\begin{proof}
Using Theorems \ref{THM: Euler characteristic} and \ref{THM: Inductive form of integral}, together with the assumption that each $\Sigma_{ij}$ is regular open, we have that: 
\begin{align*}
    2\pi \chi(\boldsymbol{\Sigma}) & = 2\pi \left(\sum_{p=1}^n (-1)^{p+1} \sum_{i_1,...,i_p \in I} \chi(\overline{\Sigma_{i_1,...,i_p}}) \right) \\
    & = \sum_{p=1}^n (-1)^{p+1} \sum_{i_1,...,i_p \in I} 2\pi \chi(\overline{\Sigma_{i_1,...,i_p}}) \\
    & =  \sum_{p=1}^n (-1)^{p+1} \sum_{i_1,...,i_p \in I} \left(\frac{1}{2} \int_{\overline{\Sigma_{i_1,...,i_p}}}RdA + \int_{\partial \Sigma_{i_1,...,i_p}} \kappa d\gamma \right) \\
    & = \sum_{p=1}^n (-1)^{p+1} \sum_{i_1,...,i_p \in I} \left(\frac{1}{2}\int_{\overline{\Sigma_{i_1,...,i_p}}}RdA  \right) + \sum_{p=2}^n (-1)^{p+1} \sum_{i_1,...,i_p \in I} \left(  \int_{\partial \Sigma_{i_1,...,i_p}} \kappa d\gamma \right) \\
    & = \frac{1}{2}\int_{\boldsymbol{\Sigma}}\textbf{R} dA + \sum_{p=2}^n (-1)^{p+1} \sum_{i_1,...,i_p \in I} \left(  \int_{\partial \Sigma_{i_1,...,i_p}} \kappa d\gamma \right), 
\end{align*}
    as required.
\end{proof}

\section{Conclusion}

The goal of this paper was to describe some basic features of de Rham cohomology for non-Hausdorff manifolds. Using the colimit description of $\textbf{M}$ already available to us, we constructed vector fields, differential forms and their integrals in Section 2. Naturally, we related these to the various Hausdorff formulas for the submanifolds $M_i$. Of central interest at this stage was the subadditivity of integration, which required a compactification of the pairwise intersections $M_{ij}$. In turn, this caused Stoke's theorem to fail in a controlled manner.  \\

De Rham cohomology itself was described in Section 3. Through an assumption of regular-open sets satisfying certain intersection properties, we saw that the de Rham cohomology for our non-Hausdorff manifold $\textbf{M}$ can be related to the cohomologies of the $M_i$ using Mayer-Vietoris sequences. In the binary case, a long exact sequence emerged, and the more-general colimit birthed a Čech-de Rham bicomplex whose cohomology coincided with de Rham by standard means. \\

When discussing singular homology in Section 4, we saw that we could derive an equivalence with the smooth theory without appealing to Whitney's embedding theorem. Instead, the equivalence followed from some derivations of Mayer-Vietoris sequence for smooth singular homology, arising from first principles. We then proved de Rham's theorem by a similar approach, this time by appealing to the standard pairing via integration over chains. \\

Finally, we combined the results of Sections 2,3 and 4 into a proof of a non-Hausdorff Gauss Bonnet theorem. After a brief discussion of curvature in $2$-manifolds, we derived a convenient subadditivity property of the Euler Characteristic. In combination with Theorem \ref{THM: Inductive form of integral}, we then proved the desired Gauss-Bonnet theorem. Importantly, the compactifications of the sets $M_{ij}$ within our definition of integration forced us to invoke the Gauss-Bonnet theorem for manifolds with boundary. In this context, these boundaries were ``internal", in the sense that they were characterised by the Hausdorff-violating submanifolds sitting inside $\textbf{M}$. As a final result, we proved this in further generality. \\

Finally, we finish the paper with a few comments regarding further work. As the astute reader may notice, throughout this paper we discuss neither compactly-supported forms nor their de Rham cohomology. There is at least some discussion of these topics within the literature, however it appears that their notion of smooth function differs from that used in this paper. Nonetheless, it will be interesting to compute the Poincaré Duality of our non-Haudorff manifolds, and compare it to that found in \cite{crainic1999remark} and \cite{kasparov1991groups}. As well as this, it will be interesting to see whether a Čech-de Rham isomorphism can be approached in a similar manner to that of Theorem \ref{THM: de Rhams theorem}. According to \cite{oconnell2023vector}, there is a Mayer-Vietoris sequence that relates the Čech cohomology of $\textbf{M}$ to the cohomologies of $M_1$ and $M_2$. Although an explicit Čech-de Rham isomorphism cannot be constructed on $\textbf{M}$ due to the non-existence of good open covers, it nonetheless seems reasonable to construct homomorphisms between the Hausdorff spaces in both bicomplexes, and then to establish an equivalence via arguments involving the associated spectral sequences. This idea can be captured within the following.
\begin{conjecture}
    Let $\textbf{M}$ be a non-Hausdorff manifold constructed according to \ref{REM: assumptions in this paper} and \ref{THM: MV long exact sequence for de rham}. Suppose furthermore that each gluing region $M_{ij}$ is regular-open. Then $$  \check{H}^q(\textbf{M}, \mathbb{R}) \cong H^q_{dR}(\textbf{M}) \cong H^q_S(\textbf{M},\mathbb{R})           $$ for all $q$ in $\mathbb{N}$.
\end{conjecture}

\subsection*{Acknowledgements}
This paper was made possible by the funding received from the Okinawa Institute of Science and Technology. I am thankful to Yasha Neiman, Slava Lysov and Andrew Lobb for the helpful discussions throughout the developmental stages of this paper, and I'd also like to thank Jeffrey Viaclovsky and Steve Halperin for the correspondence. 

\printbibliography

\appendix
\section{Appendix}

\subsection{Some results regarding Integrals}

\begin{lemma}\label{LEM: restriction of compactly supported form is still compactly supported}
    Let $\boldsymbol{\omega}$ be a compactly supported form on $\textbf{M}$. Then $\phi_i^*\boldsymbol{\omega}$ is compactly supported on $M_i$.
\end{lemma}
\begin{proof}
    Let $\boldsymbol{\omega}$ be some compactly-supported form in $\textbf{M}$. Since $\phi_i$ is continuous, the preimage of a closed set will still be closed in $M_i$. By definition, the pullback $\omega_i$ into $M_i$ will have support $$ \text{supp}(\omega_i) = \phi_i^{-1}( \text{supp}(\boldsymbol{\omega})). $$
    We now argue that this set is compact. Suppose that $\mathcal{U}$ is some open cover of the set $\text{supp}(\omega_i)$ in $M_i$. We would like to construct some open cover $\mathcal{V}$ of $\text{supp}(\boldsymbol{\omega})$ in $\textbf{M}$ such that $\mathcal{V}$ coincides with $\phi_i(\mathcal{U})$ on $M_i$. This can be achieved by considering the sets $$ \{  \phi_i(U) \ | \ U \in \mathcal{U} \} \cup \bigcup_{j\neq i}  \{ \phi_j(V) \ | \ V \cap M_{ij} = \overline{f_{ij}}(U) \text{ for some } U\in \mathcal{U}  \} \cup \bigcup_{j\neq i} Int^{\textbf{M}}(M_j \backslash M_i).  $$
    The above forms an open cover of the set $\text{supp}(\boldsymbol{\omega})$, so by compactness there exists some finite subcover. Moreover, a smaller subcover of this set will cover the set $\phi_i(\text{supp}(\omega_i))$, and we may pull this back to a finite subcover of $\mathcal{U}$. 
\end{proof}

 \begin{theorem}
        The integral of a differential form $\boldsymbol{\omega}$ over $\textbf{M}$ satisfies: 
        $$ \int_{\textbf{M}} \boldsymbol{\omega} = \sum_{i=1}^n \left(\int_{M_i} \phi_i^*\boldsymbol{\omega}\right) -  \sum_{p=2}^n(-1)^{p}\left(\sum_{\substack{ i_1,\cdots, i_p \in I \\i_1 < \cdots < i_p}} \int_{\overline{M_{i_1 \cdots i_p}}} \phi_{{i_1 \cdots i_p}}^*\boldsymbol{\omega} \right)    $$
    \end{theorem} 
    \begin{proof}
    Proceeding via induction, we may write $\textbf{M}$ as the binary adjunction $\textbf{N}\cup M_{n+1}$. Running through the same argument as Lemma \ref{THM: subadditivity for integration}, we have that $$ \int_{\textbf{M}} \boldsymbol{\omega} = \int_{\textbf{N}}\boldsymbol{\omega} + \int_{M_{n+1}} \boldsymbol{\omega} - \int_{\overline{A}} \boldsymbol{\omega}   $$
    Applying the induction hypothesis to $\textbf{N}$ and $A$ (viewed as a subspace of $M_{n+1}$) we have that 
    \begin{align*}
        \int_{\textbf{M}} \boldsymbol{\omega} & = \int_{\textbf{N}}\boldsymbol{\omega} + \int_{M_{n+1}} \boldsymbol{\omega} - \int_{\overline{A}} \boldsymbol{\omega} \\
        & = \sum_{i=1}^n \left(\int_{M_i} \phi_i^*\boldsymbol{\omega}\right) -  \sum_{p=2}^n(-1)^{p}\left(\sum_{\substack{ i_1,\cdots, i_p \\i_1 < \cdots < i_p}} \int_{\overline{M_{i_1 \cdots i_p}}} \phi_{{i_1 \cdots i_p}}^*\boldsymbol{\omega} \right)  + \int_{M_{n+1}} \boldsymbol{\omega} \\
        & \ \ \ \ - \sum_{i=1}^n \left(\int_{M_i(n+1)} \phi_i^*\boldsymbol{\omega}\right) +  \sum_{p=2}^n(-1)^{p}\left(\sum_{\substack{ i_1,\cdots, i_p \\i_1 < \cdots < i_p}} \int_{\overline{M_{i_1 \cdots i_p (n+1)}}} \phi_{{i_1 \cdots i_p(n+1)}}^*\boldsymbol{\omega} \right) \\
        & = \sum_{i=1}^{n+1} \left(\int_{M_i} \phi_i^*\boldsymbol{\omega}\right) -  \sum_{p=2}^{n+1}(-1)^{p}\left(\sum_{\substack{ i_1,\cdots, i_p \\i_1 < \cdots < i_p}} \int_{\overline{M_{i_1 \cdots i_p}}} \phi_{{i_1 \cdots i_p}}^*\boldsymbol{\omega} \right)
    \end{align*}
    as required.
    \end{proof}

\subsection{Exactness of the Mayer-Vietoris Sequence for Manifolds with Boundary}

In the proof of Theorem \ref{THM: MV long exact sequence for de rham}, we argued for the exactness of a Mayer-Vietoris sequence of the form: 
$$  0 \rightarrow \Omega^q(\overline{M_{13}}\cup \overline{M_{23}}) \rightarrow \Omega^q(\overline{M_{13}}) \oplus \Omega^q(\overline{M_{23}}) \rightarrow \Omega^q(\overline{M_{123}})\rightarrow 0        $$
where we use that $\overline{M_{123}}=\overline{M_{13}}\cap \overline{M_{23}}$.

\begin{lemma}
    Let $\overline{A}$ and $\overline{B}$ be codim-$0$ submanifolds with boundary, embedded within a Hausdorff manifold $M$. Suppose furthermore that 
    \begin{itemize}
        \item $A$ and $B$ satisfy $\overline{A} \cap \overline {B} = \overline{A\cap B}$, and 
        \item $A \cap B$ is a codim-$0$ submanifold of $M$ with boundary. 
    \end{itemize}
    Then the following is a short exact sequence. 
    $$  0 \rightarrow \Omega^q(\overline{A}\cup \overline{B}) \xrightarrow{ \ \ r \ \ } \Omega^q(\overline{A}) \oplus \Omega^q(\overline{B}) \xrightarrow{\iota_A^* - \iota_B^*} \Omega^q(\overline{A\cap B})\rightarrow 0        $$
\end{lemma}
\begin{proof}
The injectivity of $r$ and the inclusion $Im(r) \subseteq \ker{(\iota_A^* - \iota_B^*)}$ follow as in the standard Mayer-Vietoris. Moreover, an extension by zero from $\overline{A\cap B}$ into $\overline{A}$ and $\overline{B}$ will ensure the surjectivity of the inclusions $\iota_A^*$ and $\iota_B^*$.

The non-trivial step is in the proof that $\ker{(\iota_A^* - \iota_B^*)}\subseteq Im(r)$. Suppose that $\omega_A$ and $\omega_B$ are two forms on $A$ and $B$ respectively, such that $\iota_A^* \omega_A = \iota_B^* \omega_B$. We define the form $\omega$ on the union $A \cup B$ as: $$ \omega_{A\cup B}(x) = \begin{cases}
    \omega_A(x) & \textrm{if} \ x\in A \\
    \omega_B(x) & \textrm{if} \ x\in B
\end{cases}.$$ The equality $\iota_A^* \omega_A = \iota_B^* \omega_B$ guarantees that $\omega$ defines a function. We will now argue that $\omega$ is smooth. Recall first that a form on an embedded submanifold will be smooth if and only if it is the restriction of some smooth form defined on some larger open set. By assumption $\omega_A$ and $\omega_B$ are both smooth, which means that there exists 
\begin{itemize}
    \item some open set $U$ containing $A$, and a form $\omega_U \in \Omega^q(U)$ such that $\omega_U = \omega_A$ on $A$, and 
    \item some open set $V$ containing $B$, and a form $\omega_V \in \Omega^q(V)$ such that $\omega_V = \omega_B$ on $B$.
\end{itemize}
We will now glue these two forms together to define a form $\omega$ on the union $U\cup V$. In order to do so, we need to confirm that $\omega_U$ and $\omega_V$ agree on the intersection $U\cap V$. In general this may not be true, so we will now modify our forms accordingly. Let $\rho_U$ and $\rho_V$ be a partition of unity of the union $U \cup V$, subordinate to the open cover $\{U,V\}$. We now define two forms $$ \tilde{\omega}_U := \omega_U - \rho_V(\omega_U - \omega_V), \ \text{and} \  \tilde{\omega}_V := \omega_V + \rho_U(\omega_U - \omega_V).    $$
These two forms will remain unchanged on the intersection $A \cap B$, however, they will cancel their differences on the set $(U\cap V) \backslash (A\cap B)$. Indeed, on the intersection $U\cap V$ we have:
\begin{align*}
    \tilde{\omega}_U -\tilde{\omega}_V  & = \omega_U - \rho_V(\omega_U - \omega_V) - \left( \omega_V - \rho_U(\omega_U - \omega_V)\right) \\
    & = \omega_U - \omega_V - (\rho_V + \rho_U)(\omega_U - \omega_V) \\
    & = \omega_U - \omega_V - (1)(\omega_U - \omega_V) \\
    & = 0.
\end{align*}
Thus $ \tilde{\omega}_U$ and $ \tilde{\omega}_V$ agree on the intersection $U \cap V$. As such, we may glue these forms together to define a form $\omega_{U\cup V}$. By construction $\omega_{U\cup V}$ is a smooth form on the neighbourhood $U \cup V$ that restricts to $\omega_{A\cup B}$ on $A\cup B$. Thus $\omega_{A\cup B}$ is a member of $\Omega^q(A\cup B)$, and consequently $\ker{(\iota_A^* - \iota_B^*)}\subseteq Im(r)$.
\end{proof}

\begin{theorem}
Let $\{A_i\}_{i=1}^n$ be a collection of closed, codim-$0$ submanifolds of a Hausdorff $M$. Suppose furthermore that the equalities $\overline{A_{1\cdots m}} = \cap_{p=1}^m \overline{A_{p}}$ hold for all $m \leq n$. Then the sequence 
    $$      0 \rightarrow \Omega^q(A) \rightarrow \bigoplus_{i} \Omega^q(A_i) \rightarrow \bigoplus_{i<j} \Omega^q\left(\overline{A_{ij}}\right) \rightarrow \cdots \rightarrow \Omega^q\left(\overline{A_{1\cdots n}}\right) \rightarrow 0       $$
is exact for all $q$ in $\mathbb{N}$.
\end{theorem}
\begin{proof}
    Observe first that $r^*$ is injective as usual, $\delta^2=0$ via combinatoric means (found in say \cite{bott1982differential}), and the last $\delta$ map is surjective via an extension by zero argument as before. Again the non-trivial step is to ensure that $\ker(\delta)\subseteq Im(\delta)$. Let $\omega$ be some element of $\bigoplus\Omega^q\left(\overline{A_{i_0 \cdots i_p}}\right)$ such that $\delta \omega = 0$. Let $U_i$ be a collection of open sets that cover the $A_i$, and let $\rho_i$ be a partition of unity for the union $\bigcup U_i$, subordinate to the open cover $\{U_i\}$. We may extend each component $\omega_{i_0 \cdots i_p}$ into a differential form $$\tilde{\omega}_{i_0 \cdots i_p} \in \bigoplus_{i_0 < \cdots < i_p} \Omega^q(U_{i_0 \cdots i_p}).$$
    As with the previous Lemma, there is no guarantee that $\tilde{\omega}$ will be $\delta$-closed as a differential form over the $U_{i_0 \cdots i_p}$. To remedy this, we will modify each component of $\tilde{\omega}$ with a counterterm, by defining: $$  \eta_{i_0 \cdots i_p} :=\tilde{\omega}_{i_0 \cdots i_p}-  \sum_{r\neq i_0 \cdots i_p} \rho_r (\delta\omega)_{r i_0 \cdots i_p}.    $$
    Observe that each of these new forms will still restrict to $\omega_{i_0 \cdots i_p}$ on the intersections $A_{i_0 \cdots i_p}$, since the counterterm will vanish there. We now have that 
    \begin{align*}
        \delta \eta_{i_0 \cdots i_{p+1}} & = \sum_\alpha (-1)^\alpha \eta_{i_0 \cdots \hat{i_{\alpha}} \cdots i_p} \\
        & = \sum_\alpha (-1)^\alpha \left( \tilde{\omega}_{i_0 \cdots \hat{i_{\alpha}} \cdots i_p} -  \sum_{r\neq i_0 \cdots i_p} \rho_r(\delta\omega)_{r \ i_0 \cdots \hat{i_{\alpha}} \cdots i_p} \right) \\
        & = \sum_{\alpha=0}^p (-1)^\alpha \left( \tilde{\omega}_{i_0 \cdots \hat{i_{\alpha}} \cdots i_p} \right) -  \sum_\alpha (-1)^\alpha\left(\sum_{r\neq i_0 \cdots i_p} \rho_r (\delta\omega)_{r i_0 \cdots \hat{i_{\alpha}} \cdots i_p} \right). 
    \end{align*}
Unpacking this latter term, we see that there are two possible cases:
\begin{enumerate}
    \item if $r\neq i_{\hat{\alpha}}$. In this case, we see may unpack the summation of $\delta\tilde{\omega}$: 
    $$   (\delta\omega)_{r i_0 \cdots \hat{i_{\alpha}} \cdots i_p} = \sum_{\beta = 0, \cdots, p+1 \ \text{or} \beta = r} \omega_{r i_0 \cdots \hat{i_{\alpha}} \cdots \hat{i_{\alpha}} \cdots i_p}. $$ The only non-zero contribution will be when $\beta=r$, since otherwise we are omitting indices $i_{\hat{\alpha}}$ and $i_{\hat{\beta}}$ twice in the overall sum, with opposite sign.   
    \item if $r =  i_{\hat{\alpha}}$. In this case we are adding back in the deleted index, so $\rho_r (\delta\omega)_{r i_0 \cdots \hat{i_{\alpha}} \cdots i_p} = \rho_{i_{\hat\alpha}}(\delta\omega)_{i_0 \cdots \hat{i_{\alpha}} \cdots i_p}  $.
\end{enumerate}
Putting this all together, we have that 
\begin{align*}
    \sum_\alpha (-1)^\alpha\left(\sum_{r\neq i_0 \cdots i_p} \rho_r (\delta\omega)_{r i_0 \cdots \hat{i_{\alpha}} \cdots i_p} \right) & =   \sum_{\alpha = 0}^{p+1} \rho_{i_\alpha} \delta \tilde{\omega}_{i_0 \cdots i_p} + \sum_{r\neq i_0 \cdots i_{p+1}} \rho_r\delta \tilde{\omega}_{i_0 \cdots i_p}  \\
    & = \left(\sum_{i\in I} \rho_i \right)\delta\tilde{\omega}_{i_0 \cdots i_p}.
\end{align*}
We may thus conclude that $\delta \eta = 0$. Using the exactness of the Mayer-Vietoris sequence for open coverings (cf. \cite[§8]{bott1982differential}), there exists some form $\theta$ in $\bigoplus \Omega^q(U_{i_0 \cdots i_{p-1}})$ such that $\delta \theta = \eta$. The restriction of $\theta$ to the intersections $A_{i_0 \cdots i_{p-1}}$ will then yield a smooth $q$-form in $\bigoplus \Omega^q(A_{i_0 \cdots i_{p-1}})$ that maps to $\omega$ under $\delta$. This confirms that $\ker(\delta)\subseteq Im(\delta)$, from which the result follows.
\end{proof}

\end{document}